\documentclass[10pt,letterpaper]{article}
\usepackage[left=1in,top=1in,right=1in,bottom=1in]{geometry}
\usepackage{setspace}
\usepackage{amsfonts}
\usepackage{amsmath,amsthm,amssymb}
\usepackage{mathrsfs}
\usepackage{soul}
\usepackage{bbm}
\usepackage{color}
\usepackage[dvipsnames]{xcolor}
\usepackage{graphicx}
\usepackage{tikz}

\usepackage[backend=biber,style=alphabetic,sorting=nyt]{biblatex}
\newcommand{\calA}{\mathcal{A}}
\newcommand{\calB}{\mathcal{B}}

\newcommand{\calD}{\mathcal{D}}

\newcommand{\calF}{\mathcal{F}}
\newcommand{\calG}{\mathcal{G}}
\newcommand{\calH}{\mathcal{H}}

\newcommand{\calS}{\mathcal{S}}
\newcommand{\calT}{\mathcal{T}}

\newcommand{\calX}{\mathcal{X}}

\newcommand{\repn}[1]{1,\ldots,#1}

\newcommand{\repdc}[3]{#1_{#2} , \ldots , #1_{#3}}
\newcommand{\N}{{\mathbb N}}
\newcommand{\Z}{{\mathbb Z}}

\newcommand{\R}{{\mathbb R}}
\newcommand{\C}{{\mathbb C}}
\newcommand{\bs}{\backslash}

\newcommand{\I}[1]{\mathbbm{1}_{#1}}
\newcommand{\SUM}[1]{\sum_{\begin{smallmatrix}#1\end{smallmatrix}}}

\newcommand{\dif}{\textrm{d}}

\newcommand{\E}[1]{{\mathbb E}\left(#1\right)}

\newcommand{\Cov}[1]{\textnormal{Cov}\left(#1\right)}

\newcommand{\by}{\times}
\newcommand{\Mat}[2]{\textnormal{M}_{#1}\left(#2\right)}
\newcommand{\MAT}{\textnormal{M}}

\newcommand{\diag}[1]{\textnormal{diag}\left(#1\right)}

\newcommand{\Tr}[1]{\textnormal{Tr}\left(#1\right)}
\newcommand{\lTr}{\textnormal{Tr}}

\newcommand{\T}{\mathrm{T}}
\newcommand{\bsmatrix}{\left(\begin{smallmatrix}}
\newcommand{\esmatrix}{\end{smallmatrix}\right)}

\newcommand{\id}{{\bf 1}}
\newcommand{\FI}{\varphi}
\newcommand{\F}[1]{\FI\left(#1\right)}
\newcommand{\NC}{\textnormal{NC}}
\newcommand{\SP}{\textnormal{SP}}
\newcommand{\DP}{\textnormal{DP}}
\newcommand{\AP}{\textnormal{AP}}

\newcommand{\EOP}{E}
\newcommand{\Eop}[1]{\EOP\left(#1\right)}

\newcommand{\eq}[1]{\begin{equation*}
#1
\end{equation*}}
\newcommand{\eqn}[2]{\begin{equation}
\label{#1}
#2
\end{equation}}
\newcommand{\al}[1]{\begin{align*}
#1
\end{align*}}
\newcommand{\aln}[1]{\begin{align}
#1
\end{align}}
\newcommand{\dsty}[1]{$\displaystyle #1$}
\newcommand{\ndsty}[1]{$#1$}

\newtheoremstyle{DefinitionStyle}  % name of the style to be used
	{9pt}	   % measure of space to leave above the theorem. E.g.: 3pt
	{9pt}      % measure of space to leave below the theorem. E.g.: 3pt
	{\it}      % name of font to use in the body of the theorem
	{0pt}      % measure of space to indent
	{\bfseries}% name of head font
	{.}        % punctuation between head and body
	{3pt}      % space after theorem head
	{}         % Manually specify head
\theoremstyle{DefinitionStyle}
\newtheorem{theorem}{Theorem}

\newtheorem{definition}[theorem]{Definition}
\newtheorem{example}[theorem]{Example}
\newtheorem{lemma}[theorem]{Lemma}
\newtheorem{proposition}[theorem]{Proposition}
\newtheorem{corollary}[theorem]{Corollary}

\newcommand{\minus}{-}
\newcommand{\plus}{+}

\author{Mario Diaz\thanks{Correspondence to be sent to Mario Diaz: \textit{13madt@queensu.ca}\newline The work of Mario Diaz was partially supported by CIMI (Centre International de Math\'{e}matiques et d'Informatique) Excellence Program, ANR-11-LABX-0040-CIMI within the program ANR-11-IDEX-0002-02, while visiting the Institute of Mathematics of Toulouse, the Natural Sciences and Engineering Research Council of Canada, and an Ontario Trillium Scholarship.}\\{\footnotesize Queen's University} \and James Mingo\footnote{Research supported by a Discovery Grant from the Natural Sciences and Engineering Research Council of Canada.}\\{\footnotesize Queen's University} \and Serban Belinschi\\{\footnotesize CNRS - Institute of Mathematics of Toulouse}}
\title{On the Global Fluctuations of Block Gaussian Matrices}
\date{\today}

\begin{document}
\maketitle

\abstract{In this paper we study the global fluctuations of block Gaussian matrices within the framework of second-order free probability theory. In order to compute the second-order Cauchy transform of these matrices, we introduce a matricial second-order conditional expectation and compute the matricial second-order Cauchy transform of a certain type of non-commutative random variables. As a by-product, using the linearization technique, we obtain the second-order Cauchy transform of non-commutative rational functions evaluated on selfadjoint Gaussian matrices.}

\section{Introduction and main results}
\label{Section:Introduction}

Voiculescu's free probability theory \cite{Voiculescu1985} has proved to be very useful in the study of the asymptotic behavior of random matrices \cite{Voiculescu1995,HaagerupThorbjornsen2005,Anderson2013,BelinschiMaiSpeicher2015}. Among other reasons, this is the case because in many situations the expectation of the trace of products of random matrices converges to the expectation of products of non-commutative random variables. In this non-commutative setting, the so-called semicircular variables play an outstanding role. On one hand, they are the non-commutative counterpart of Gaussian random variables in many aspects; on the other, they encode the asymptotic behavior of selfadjoint Gaussian matrices.

An important aspect of free probability theory comes from its combinatorial facet, with the non-crossing partitions in its heart \cite{NicaSpeicher2006}. From this combinatorial point of view, the expectation of products of semicircular variables depend on the so-called non-crossing pairings. About a decade ago, Mingo and Nica \cite{MingoNica2004} extended this combinatorial treatment to the study of the global fluctuations (i.e., the covariance of two traces) of selfadjoint Gaussian matrices. In particular, they demonstrated that these global fluctuations depend on another type of pairings, the non-crossing annular pairings. In order to systematize the combinatorial treatment of the global fluctuations of random matrices, Mingo and Speicher introduced a theory called second-order free probability \cite{MingoSpeicher2006,MingoSniadySpeicher2007,CollinsMingoSniadySpeicher2007}. In this theory, the first- and second-order behaviors of selfadjoint Gaussian matrices are encoded in the so-called second-order semicircular variables.

The combinatorial approach based on non-crossing partitions extends to the so-called operator-valued free probability theory \cite{Speicher1998}. In this extension of free probability, the expectation is replaced by a conditional expectation that may take values in a non-commutative complex algebra. The role of semicircular variables is then played by the so-called operator-valued semicircular elements. When the conditional expectation takes values in spaces of complex matrices, operator-valued semicircular elements can be used to describe the asymptotic behavior of block Gaussian matrices \cite{RashidiFarOrabyBrycSpeicher2008}.

It is the purpose of this work to study the global fluctuations of block Gaussian matrices. In order to do so, we introduce a second-order conditional expectation taking values in spaces of complex matrices. Motivated by the previous discussion, we use the phrase
\textit{matricial second-order semicircular elements}
to refer to  
the variables that encode the behavior of block Gaussian matrices. It is important to remark that at the moment we do not have a general operator-valued second-order free probability theory as we only deal with conditional expectations taking values in spaces of complex matrices. Nonetheless, matricial conditional expectations have proved to be very valuable in many random matrix theory applications, e.g., \cite{RashidiFarOrabyBrycSpeicher2008,BelinschiMaiSpeicher2015,DiazPerezAbreu2017}, as well as in more theoretical developments, e.g., \cite{NicaShlyakhtenkoSpeicher2002,HeltonRashidiFarSpeicher2007}. Before discussing our main results, let us consider the following.

\begin{definition}
\label{Def:BlockGaussianMatrices}
Let $d\in\N$ be fixed and let $\sigma:\{1,\ldots,d\}^2\times\{1,\ldots,d\}^2\to\R$ be a given covariance mapping. A selfadjoint $dN\by dN$ random matrix $X_N$ is called a block Gaussian matrix with covariance $\sigma$ if
\eq{X_N=\left(\begin{matrix}X_N^{(1,1)} & X_N^{(1,2)} & \cdots & X_N^{(1,d)}\\ \vdots & \vdots & \ddots & \vdots\\ X_N^{(d,1)} & X_N^{(d,2)} & \cdots & X_N^{(d,d)}\end{matrix}\right)}
where $\{X_N^{(p,q)} : 1\leq p,q\leq d\}$ are $N\by N$ selfadjoint random matrices such that
\eq{\{\Re(X_N^{(p,q)}(i,j)),\Im(X_N^{(p,q)}(i,j)) : 1\leq p,q\leq d,1\leq i,j\leq N\}}
are jointly Gaussian with zero mean and covariance specified by
\eq{\E{X_N^{(p,q)}(k',l')X_N^{(r,s)}(k'',l'')} = \frac{1}{N} \delta_{k',l''} \delta_{l',k''} \sigma(p,q;r,s),}
where $\delta_{k,l}$ equals 1 if $k=l$ and 0 otherwise.
\end{definition}

Note that the Gaussian Unitary Ensemble (GUE) corresponds to the case $d=1$. For each $n\in\N$, let
\eq{\alpha_n = \lim\limits_{N\to\infty} N^{-1}\E{\Tr{X_N^n}}.}
These first-order moments are encoded in the generating function
\eqn{eq:DefG}{G(z) = \sum_{n\geq0} \frac{\alpha_n}{z^{n+1}}.}
We call this generating function the Cauchy transform of the first-order moments of $X_N$, however here we are not concerned with its representation as the Cauchy transform of a measure. Using a conditional expectation taking values in the algebra of $d\by d$ complex matrices, Helton et al. proved the following \cite{HeltonRashidiFarSpeicher2007}. Let $\MAT_d$ be the algebra of $d\by d$ complex matrices and $\eta:\MAT_d\to\MAT_d$ be given by
\eq{\eta(w)(p,q) = \sum_{k,l=1}^d w(k,l) \sigma(p,k;l,q),}
for all $w \in \MAT_d$ and $1\leq p,q\leq d$. For a given $z\in\C$ with $\Im(z)>0$, let $T_z:\MAT_d\to\MAT_d$ be determined by
\eq{T_z(w) = (z{\rm I}_d-\eta(w))^{-1}.}
If $|z|$ is large enough, then the right hand side of \eqref{eq:DefG} converges absolutely and
\eqn{eq:GTrG}{G(z)= d^{-1} \Tr{\calG(z)},}
where $\calG:\C^+\to\MAT_d$ is an analytic function given by
\eqn{eq:FixedPointEquationcalG}{\calG(z) = \lim_{n\to\infty} T_z^{\circ n}(w)}
for any $w\in\MAT_d$ with $\Im w < 0$ (see \cite{HeltonRashidiFarSpeicher2007}). Note that $\calG(z) = T_z(\calG(z))$.

In this paper we extend the previous analysis to the global fluctuations of block Gaussian matrices. Specifically, for each $m,n\in\N$, let
\eq{\alpha_{m,n} = \lim_{N\to\infty} \Cov{\Tr{X_N^m},\Tr{X_N^n}}.}
These second-order moments are encoded in the generating function
\eqn{eq:DefG2}{G_2(z,w) = \sum_{m\geq1} \sum_{n\geq1} \frac{\alpha_{m,n}}{z^{m+1}w^{n+1}}.}
As before, we call this generating function the second-order Cauchy transform of $X_N$. The main result of this paper provides a closed form expression for $G_2$ in terms of the mapping $\calG$. Let $\Sigma\in\MAT_d\otimes\MAT_d$ be the matrix given by $\Sigma(p,q;r,s) = \sigma(p,q;r,s)$, see Section~\ref{Subsection:Matrices}. Also, for every $A\in\MAT_d$ let $\Theta(A)(p,q;r,s)=A(p,r;q,s)$, $A^\Gamma(p,q;r,s)=A(p,q;s,r)$, and $\Phi(A)=\Theta(A^\Gamma)$.

\begin{theorem}
Let $X_N$ be a block Gaussian matrix with covariance $\sigma$. There exists $K\in\R_+$ such that if $z,w\in\C$ with $|z|,|w|>K$, then the right hand side of \eqref{eq:DefG2} converges absolutely and
\eq{G_2(z,w) = \Tr{\calG_2(z,w)},}
where $\calG_2:(\C\bs[-K,K])^2\to\MAT_{d^2}$ is an analytic function given by
\eqn{eq:MainTheorem}{\calG_2(z,w) = \Theta\left(\calG_D(z,z) \left\{\Theta[\calH(z,w)]+\Phi\left[\calH(z,w)^\Gamma(\calG(z)\otimes\calG(w))^\Gamma\calH(z,w)^\Gamma\right]\right\} \calG_D(w,w)^\T\right)}
with
\al{
\calH(z,w) &=\left({\rm I}_{d^2}-\Sigma[\calG(z)\otimes\calG(w)]\right)^{-1}\Sigma,\\
\calG_D(z,w) &= [\calG(z)\otimes\calG(w)] \left({\rm I}_{d^2}-\Sigma [\calG(z)\otimes\calG(w)]\right)^{-1}.
}
\end{theorem}

\begin{proof}
This theorem is a direct consequence of Theorem~\ref{Thm:OpVSOCauchyTransform} and Corollary~\ref{Corollary:AnalyticitySOCauchyTransform} below.
\end{proof}

When $d=1$ and $\sigma(1,1;1,1)=1$, equation \eqref{eq:MainTheorem} becomes the much simpler equation
\eq{G_2(z,w) = \frac{G(z)^2G(w)^2}{[1-G(z)^2][1-G(w)^2][1-G(z)G(w)]^2}.}
Since the Cauchy transform of the semircircle distribution satisfies that \ndsty{-zG'(z)[1-G(z)^2] = G(z)^2}, the previous equation is equivalent to
\eq{G_2(z,w) = \frac{zwG'(z)G'(w)}{[1-G(z)G(w)]^2}.
}
This equation can be found in an unpublished work of Mingo and Nica. Thus, our main theorem can be regarded as a generalization of this formula beyond the case $d=1$. Another important feature of Mingo and Nica's formula is that $G_2$ can be computed from $G$. In a similar way, equation \eqref{eq:MainTheorem} readily shows that $\calG_2$ can be obtained from $\calG$.

On a combinatorial level, the proof of our main theorem relies on three types of pairings: single-line, double-line, and annular pairings. Single-line pairings are the usual non-crossing pairings in free probability theory \cite{NicaSpeicher2006}. Annular pairings are the pairings introduced by Mingo and Nica to study the fluctuations of selfadjoint Gaussian matrices \cite{MingoNica2004}. To the best of the authors's knowledge, this is the first time that double-line pairings, and their generating function, appear in the free probability literature. Using the relations between these three types of pairings, equation \eqref{eq:MainTheorem} is established at the level of formal expressions and then extended to an analytic level.

By definition, the matricial second-order Cauchy transform $\calG_2$ takes matricial arguments, allowing us to establish \eqref{eq:MainTheorem} in greater generality, cf. Theorem~\ref{Thm:OpVSOCauchyTransform} and Theorem~\ref{Thm:AnalyticalBehavior}. As a by-product, it is possible to obtain the second-order Cauchy transform of a non-commutative polynomial evaluated on selfadjoint Gaussian matrices. Let $q=q^*$ be a polynomial in $r$ non-commutative indeterminates $\langle \repdc{x}{1}{r} \rangle$. For notational simplicity, let $x_0 = 1$. Assume that $X_{1,N},\ldots,X_{r,N}$ are $N\by N$ independent selfadjoint Gaussian matrices and let $Q_N := q(X_{1,N},\ldots,X_{r,N})$. The second-order Cauchy transform of $Q_N$ is given by
\eq{G_2^q(z,w) = \sum_{m\geq1} \sum_{n\geq1} \frac{\alpha^q_{m,n}}{z^{m+1}w^{n+1}},}
where \dsty{\alpha^q_{m,n} = \lim_{N\to\infty} \Cov{\Tr{Q_N^m},\Tr{Q_N^n}}}. By the so-called linearization technique \cite{HaagerupThorbjornsen2005,Anderson2013}, there exist $d\in\N$ and selfadjoint $d\by d$ matrices $\repdc{A}{0}{r}$ such that the linear pencil
\eq{L(x_1,\ldots,x_r)=\sum_{i=0}^r A_i \otimes x_i}
satisfies
\eq{(z-q(\repdc{x}{1}{r}))^{-1} = (\Lambda_z-L(x_1,\ldots,x_r))^{-1}_{1,1} ,}
where $\Lambda_z = \diag{z,0,\ldots,0}$. By the definitions of the second-order conditional expectation and the matricial second-order Cauchy transform introduced in Section~\ref{Section:Fluctuations}, we have that at the level of formal expressions
\eqn{eq:RelationSOCTPolynomials}{G^q_2(z,w) = \calG^X_2(\Lambda_z-A_0,\Lambda_w-A_0)(1,1;1,1),}
where $X$ is the matricial second-order semicircular element associated to
\eq{L(X_{1,N},\ldots,X_{r,N}) - A_0 \otimes {\rm I}_N = A_1 \otimes X_{1,N} + \cdots + A_r \otimes X_{r,N}.}
The linearization technique can also be applied to non-commutative rational functions, see, e.g., \cite{HeltonMaiSpeicher2017}. Hence, the relation in \eqref{eq:RelationSOCTPolynomials} holds true, mutatis mutandis, whenever $q=q^*$ is a non-commutative rational function (with a suitable domain).

It is important to remark that there are many key contributions to the study of the global fluctuations of not necessarily (block) Gaussian random matrices. For example: Diaconis and Shahshahani's work on the unitary, orthogonal, symplectic, and symmetric groups \cite{DiaconisShahshahani1994}; Johansson's study of Hermitian matrices with eigenvalue distributions determined by a potential \cite{Johansson1998}; Bai and Silverstein's work on sample covariance matrices \cite{BaiSilverstein2004}, and further extensions \cite{HachemLoubatoNajim2008,HachemKharoufNajimSilverstein2012}; and Anderson and Zeitouni's contributions to the case of band matrix models \cite{AndersonZeitouni2006}.

The paper is organized as follows. In the following section we gather some preliminary results and notation needed through the paper. The matricial second-order analogues of the moments, Cauchy transform, and semicircular elements are introduced in Section~\ref{Section:Fluctuations}. In Section~\ref{Section:FluctuationsBlockRandomMatrices} we establish the connection between the global fluctuations of block Gaussian matrices and the matricial second-order semicircular elements. In Section~\ref{Section:Pairings} we establish the properties of single-line, double-line, and annular pairings required to establish, in Section~\ref{Section:CauchyTransforms}, the combinatorial expressions for the Cauchy transforms associated to these pairings.  Finally, in Section~\ref{Section:AnalyticProperties}, we extend these formulas to an analytic level.

\section{Preliminaries and notation}
\label{Section:Preliminaries}

\subsection{General notation}

For $n\geq1$, we let $[n]=\{\repn{n}\}$. Given $\repdc{i'}{1}{m}\in[d]$, we let $i':[m]\to[d]$ be the function determined by $i'(s)=i'_s$. For $i':[m]\to[d]$ and $i'':[n]\to[d]$, we let $i:[m+n]\to[d]$ be the function given by
\eq{i(s)=i'_s\I{s\leq m}+i''_{s-m}\I{s>m},}
where $\I{}$ is the indicator function. For notational simplicity, we use $i_s$ instead of $i(s)$. Whenever $d$, $m$ and $n$ are clear from the context, we use the notation
\eq{\sum\limits_{i,j} \text{ to represent the sum } \sum_{\repdc{i'}{1}{m}=1}^d \sum_{\repdc{j'}{1}{m}=1}^d \sum_{\repdc{i''}{1}{n}=1}^d \sum_{\repdc{j''}{1}{n}=1}^d.}
For a given set $\calX$, we let $\calX^n = \calX \times \cdots \times \calX$ be its $n$-fold Cartesian product. We assume that all classical random variables belong to a probability space $(\Omega,\calF,\mathbb{P})$ with expectation $\mathbb{E}$. For random variables $\repdc{X}{1}{r}$, we let $k_r(\repdc{X}{1}{r})$ denote their classical $r$-th cumulant. In particular, for random variables $X$ and $Y$, we have that
\eq{k_1(X)=\E{X} \quad \text{ and } \quad k_2(X,Y)=\E{XY}-\E{X}\E{Y}.}

\subsection{Partitions and permutations}

A partition of a non-empty set $V$ is a family $\pi=\{\repdc{V}{1}{S}\}$ such that $\emptyset \neq V_s \subset V$ for all $s\in [S]$, $V_s\cap V_t = \emptyset$ for $s\neq t$, and $\bigcup_{s\in [S]} V_s = V$. We say that $\pi$ is a partition of $n$ points if $V=[n]$. For notational convenience, we define $\emptyset$ as the only partition of 0 points. For an integer $l$ and a partition $\pi=\{\repdc{V}{1}{S}\}$ of a non-empty set $V\subset\Z$, we let $\pi+l$ be the partition of $\{v+l:v\in V\}$ given by
\eq{\pi+l = \big\{ \{v+l : v\in V_s\} : s\in[S]\big\}.}
We let $\emptyset+l:=\emptyset$ for all $l\in\Z$. Let $\pi=\{\repdc{V}{1}{S}\}$ be a partition of a non-empty set $V$; the relation $u \sim_\pi v$ whenever $u,v\in V_s$ for some $s\in[S]$ defines an equivalence relation on $V$; for $W\subset V$, we let $\pi|_{W}$ be the partition of $W$ such that for every $u,v\in W$ we have that $u\sim_{\pi|_W} v$ if and only if $u\sim_{\pi} v$. A non-crossing partition of a non-empty set $V\subset\Z$ is a partition $\pi$ of $V$ such that if $a \sim_\pi b$, $c \sim_\pi d$, and $a<c<b<d$, then $a \sim_\pi b \sim_\pi c \sim_\pi d$. We denote by $\NC(n)$ the set of all non-crossing partitions of $n$ points. We say that a partition $\pi=\{\repdc{V}{1}{S}\}$ is a pairing if $|V_s|=2$ for all $s\in[S]$. We denote by $\NC_2(n)$ the set of all non-crossing pairings of $n$ points. By definition, the partition $\emptyset$ is a non-crossing pairing. Given a pairing of $n$ points $\pi$, we let $\hat{\pi} = \{\{n+1-u,n+1-v\} : \{u,v\}\in\pi\}$.

Let $\calS_V$ be the group of permutations of a non-empty set $V$. For a permutation $\pi$, we let $\#(\pi)$ denote the number of cycles of $\pi$. For $m,n\in\N$, we define the permutation $\gamma_{m,n}:[m+n]\to[m+n]$ by
\eq{\gamma_{m,n}(p) = \begin{cases}p+1 & p\neq m \textnormal{ and } p\neq m+n,\\1 & p=m,\\m+1 & p=m+n.\end{cases}}
By abuse of notation, given a pairing $\pi$ of $V$ we let $\pi:V\to V$ be the permutation of V such that $\pi(u) = v$ whenever $v\neq u$ and $v \sim_\pi u$. A pairing of $m+n$ points $\pi$ is called a $(m,n)$-annular non-crossing pairing if $\#(\gamma_{m,n} \pi) = \frac{m+n}{2}$ and there exists $u,v\in[m+n]$ such that $u \sim_\pi v$ and $u\leq m < v$. If we draw two concentric circles, the exterior one with $m$ points labelled clockwise and the interior one with $n$ points labelled counterclockwise, then the $(m,n)$-annular non-crossing pairings correspond to the pairings of these points that can be drawn without crossings and that have at least one string connecting both circles. Figure~\ref{Fig:ExAnnularPairing} depicts the $(4,4)$-annular non-crossing pairing $\pi=\{\{1,6\},\{2,5\},\{3,4\},\{7,8\}\}$. We denote by $\NC_2(m,n)$ the set of all $(m,n)$-annular non-crossing pairings.

\begin{figure}[ht]
\centering
\includegraphics[scale=0.3]{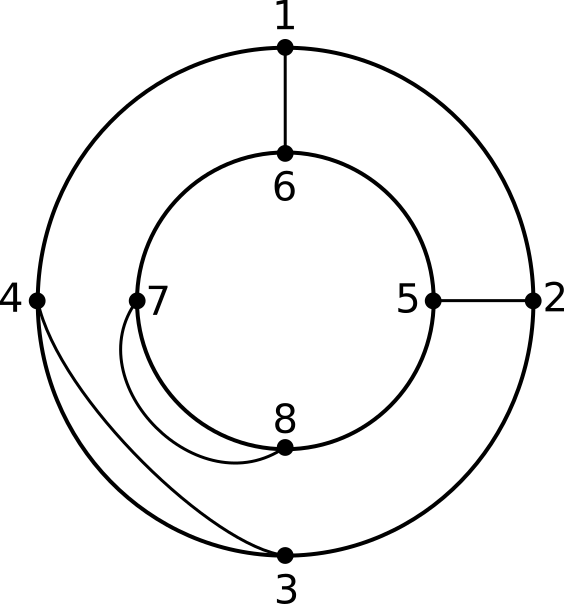}
\caption{Graphical representation of a $(4,4)$-annular non-crossing pairing.}
\label{Fig:ExAnnularPairing}
\end{figure}

\subsection{Matrices}
\label{Subsection:Matrices}

Let $\Mat{m\times n}{\calA}$ denotes the set of all $m\times n$ matrices with entries in $\calA$. For notational simplicity, we write $\Mat{n}{\calA}$ instead of $\Mat{n\times n}{\calA}$. When $\calA=\C$, we simply write $\MAT_{m\by n}$ and $\MAT_n$. For a matrix $A\in\Mat{m\times n}{\calA}$, we let $A(p,q)$ be the $p,q$-entry of $A$. We denote by $\lTr:\Mat{n}{\calA}\to\calA$ the trace function given by \dsty{\Tr{A} = \sum_{k=1}^n A(k,k)}. Also, we let $\|A\|$ be the operator norm of $A\in\MAT_n$. For $A\in\MAT_{m'\times n'}$ and $B\in\MAT_{m''\times n''}$, we let $A\otimes B$ be the $m'm'' \times n'n''$ matrix given by
\eq{A\otimes B = \left(\begin{matrix}A(1,1)B & A(1,2)B & \cdots & A(1,n')B\\A(2,1)B & A(2,2)B & \cdots & A(2,n')B\\ \vdots & \vdots & \ddots & \vdots\\ A(m',1)B & A(m',2)B & \cdots &A(m',n')B\end{matrix}\right).}
For $X\in\MAT_{d^2}$, we let \ndsty{X(p,q;r,s) = X((p-1)d+r,(q-1)d+s)} for all $p,q,r,s\in[d]$. In particular, if $A,B\in\MAT_d$, then
\eqn{eq:TensorProductpqrs}{(A\otimes B)(p,q;r,s) = A(p,q)B(r,s).}
With this notation, if $X,Y\in\MAT_{d^2}$ then, for $p,q,r,s\in[d]$,
\eqn{eq:SumTensorProduct}{(XY)(p,q;r,s) = \sum_{k,l=1}^d X(p,k;r,l)Y(k,q;l,s).}

Assume that $X_N^{(p,q)}$ is an $N\by N$ (random) matrix for each $p,q\in[d]$. We let $X_N=(X_N^{(p,q)})_{p,q}$ be the $dN\by dN$ (random) matrix given by
\eqn{eq:ConventionBlockMatrix}{X_N = \left(\begin{matrix}X_N^{(1,1)} & X_N^{(1,2)} & \cdots & X_N^{(1,d)}\\X_N^{(2,1)} & X_N^{(2,2)} & \cdots & X_N^{(2,d)}\\\vdots & \vdots & \ddots & \vdots\\X_N^{(d,1)} & X_N^{(d,2)} & \cdots & X_N^{(d,d)}\end{matrix}\right).}
In particular, $X_N(p,q;r,s)=X_N^{(p,q)}(r,s)$. Identifying $\MAT_{dN}\cong\MAT_d\otimes\MAT_N$, we let
\eq{(\id_{\MAT_d}\otimes\lTr_{\MAT_N})(X_N) = \left(\Tr{X_N^{(p,q)}}\right)_{p,q=1}^d\in\MAT_d}
where $\id_{\MAT_d}$ and $\lTr_{\MAT_N}$ are the identity on $\MAT_d$ and the trace on $\MAT_N$, respectively. When there is no risk of confusion, we write $\id\otimes\lTr$ without any reference to the specific spaces.

\subsection{Second-order free probability}
\label{Subsection:SOFreeProbability}

A second-order (non-commutative) probability space $(\calA,\FI,\FI_2)$ consists of a unital algebra $\calA$, a tracial linear functional $\FI:\calA\to\C$ with $\F{1}=1$, and a bilinear functional $\FI_2:\calA\times\calA\to\C$ which is tracial in both arguments and satisfies $\FI_2(a,1) = \FI_2(1,a) = 0$ for all $a\in\calA$. For non-commuting random variables $x_1,\ldots,x_N\in\calA$ and $i_1,\ldots,i_m\in[N]$, we let \ndsty{\prod_{w=1}^m x_{i_w}} denote \ndsty{x_{i_1} x_{i_2} x_{i_3} \cdots x_{i_m}} in that specific order.

\begin{definition}
\label{Def:SOSemicircularFamily}
Let $(\calA,\FI,\FI_2)$ be a second-order probability space. We say that a family of selfadjoint operators $x_1,\ldots,x_N\in\calA$ is a second-order semicircular family with covariance $\sigma:[N]\times[N]\to\R$ if for all $m,n\in\N$, $i':[m]\to[N]$ and $i'':[n]\to[N]$ we have that
\aln{
\nonumber \F{\prod_{w=1}^m x_{i'_w}} &= \sum_{\pi\in\NC_2(m)} \prod_{\underset{\pi}{u \sim v}} \sigma(i'_u;i'_v),\\
\label{eq:SOMomentsSOSF} \FI_2\left(\prod_{w=1}^m x_{i'_w},\prod_{w=1}^n x_{i''_w}\right) &= \SUM{\pi\in\NC_2(m,n)} \prod_{\underset{\pi}{u \sim v}} \sigma(i_u;i_v).
}
\end{definition}

\begin{definition}
\label{Def:ConvergenceSODistribution}
For each $N\in\N$, let $\{A_N^{(1)},\ldots,A_N^{(S)}\}$ be a collection of $N\by N$ random matrices. Assume that $(\calA,\FI,\FI_2)$ is a second-order probability space and $\repdc{a}{1}{S}\in\calA$. We say that $\{A_N^{(1)},\ldots,A_N^{(S)}\}$ converges in second-order distribution to $\{\repdc{a}{1}{S}\}$, denoted by $\{A_N^{(1)},\ldots,A_N^{(S)}\} \stackrel{\textnormal{so-dist}}{\longrightarrow} \{\repdc{a}{1}{S}\}$, if for all polynomials $p_1,p_2,\ldots$ in $S$ non-commuting indeterminates we have
\al{
\lim_{N\to\infty} \E{\frac{1}{N} \lTr\left[p_1\big(A_N^{(1)},\ldots,A_N^{(S)}\big)\right]} &= \F{p_1(\repdc{a}{1}{S})},\\
\lim_{N\to\infty} k_2\left(\lTr\left[p_1\big(A_N^{(1)},\ldots,A_N^{(S)}\big)\right],\lTr\left[p_2\big(A_N^{(1)},\ldots,A_N^{(S)}\big)\right]\right) &= \FI_2(p_1(\repdc{a}{1}{S}),p_2(\repdc{a}{1}{S})),
}
and, for $r\geq3$,
\eq{\lim_{N\to\infty} k_r\left(\lTr\left[p_1\big(A_N^{(1)},\ldots,A_N^{(S)}\big)\right],\ldots,\lTr\left[p_r\big(A_N^{(1)},\ldots,A_N^{(S)}\big)\right]\right) = 0.}
\end{definition}

\section{MSO moments and the MSO Cauchy transform}
\label{Section:Fluctuations}

Assume that $(\calA,\FI,\FI_2)$ is a second-order non-commutative probability space. Let $\Mat{d}{\calA}$ be the algebra of $d\by d$ matrices over $\calA$. A natural conditional expectation $\EOP:\Mat{d}{\calA}\to\MAT_d$ is given by
\eq{\Eop{X}(p,q) = \F{X(p,q)},}
for all $X\in\Mat{d}{\calA}$ and $p,q\in[d]$. Motivated by the previous equation, in this section we introduce a {\it conditional} version of $\FI_2$ which leads to the notions of matricial second-order moments and matricial second-order Cauchy transform. As shown in the following section, they can be used to obtain the second-order moments and Cauchy transforms of certain block random matrices. In particular, the second-order behavior of the block Gaussian matrices introduced in Section~\ref{Section:Introduction} can be obtained from the matricial second-order semicircular elements defined below. The main result of this section establishes that the matricial second-order Cauchy transform associated to matricial second-order semicircular elements admits a description in terms of the non-crossing annular pairings.

Recall that $\FI_2$ is bilinear on $\calA\times\calA$, so it can be naturally extended to $\calA\otimes\calA$. We define $\EOP_2:\Mat{d}{\calA}\otimes\Mat{d}{\calA}\to\MAT_d\otimes\MAT_d$ by
\eq{\EOP_2(X\otimes Y) = (\id_{\MAT_d\otimes\MAT_d}\otimes\FI_2)(X\otimes Y),}
for all $X,Y\in\Mat{d}{\calA}$. In the previous equation, we identify $\Mat{d}{\calA}\otimes\Mat{d}{\calA}$ with $(\MAT_d\otimes\MAT_d)\otimes(\calA\otimes\calA)$. In the convention from Section~\ref{Section:Preliminaries}, the previous equation reads as
\eq{\EOP_2(X\otimes Y)(p,q;r,s) = \FI_2(X(p,q),Y(r,s)),}
for all $p,q,r,s\in[d]$. Note that $\EOP_2(X\otimes Y)$ is a $d^2\by d^2$ matrix, in contrast to $\Eop{X}$ which is a $d\by d$ matrix. Observe that, for all $a\in\MAT_d$ and all $X\in\Mat{d}{\calA}$,
\eqn{eq:SOConditionalExpectationI}{\EOP_2(a\otimes X)=\EOP_2(X\otimes a)=0.}
Also, since $\FI_2$ is bilinear, equations \eqref{eq:TensorProductpqrs} and \eqref{eq:SumTensorProduct} readily imply that
\eqn{eq:SOConditionalExpectationII}{\EOP_2(aXb\otimes cYd) = (a\otimes c) \EOP_2(X\otimes Y) (b\otimes d)}
for all $a,b,c,d\in\MAT_d$ and $X,Y\in\Mat{d}{\calA}$. Features \eqref{eq:SOConditionalExpectationI} and \eqref{eq:SOConditionalExpectationII} make $\EOP_2$ a second-order analogue of the conditional expectation $\EOP$: $\Eop{a}=a$ and $\Eop{aXb} = a\Eop{X}b$ for all $a,b\in\MAT_d$ and $X\in\Mat{d}{\calA}$. 

\begin{definition}
Let $X\in\Mat{d}{\calA}$. We define the matricial second-order (MSO) $(m,n)$-moment of $X$ by
\eq{M_{m,n}({\bf a};{\bf b}) = \EOP_2(a_0Xa_1\cdots Xa_m \otimes b_0Xb_1\cdots Xb_n)}
where ${\bf a}=(a_0,\ldots,a_m)\in\MAT_d^{m+1}$ and ${\bf b}=(b_0,\ldots,b_n)\in\MAT_d^{n+1}$. By abuse of notation, we let $M_{m,n}(a;b)$ denote $M_{m,n}(a,\ldots,a;b,\ldots,b)$ for every $a,b\in\MAT_d$.

We define the matricial second-order (MSO) Cauchy transform of $X$ by
\eqn{eq:DefMSOCT}{\calG_2(a,b) = \sum_{m,n\geq1} M_{m,n}(a^{-1};b^{-1})}
for $a,b\in\MAT_d$ invertible.
\end{definition}

Equation~\eqref{eq:DefMSOCT} should be interpreted at the level of formal expressions. In other words, we think of $\calG_2(a,b)$ as an element of
\eq{\Mat{d^2}{\C[[\{(a^{-1})(i,j),(b^{-1})(i,j) : i,j\in[d]\}]]},}
the algebra of $d^2 \by d^2$ matrices over the formal power series in the $2d^2$ commuting variables
\eq{\{(a^{-1})(i,j),(b^{-1})(i,j) : i,j\in[d]\}.}
We only evaluate \eqref{eq:DefMSOCT} whenever $a$ and $b$ are invertible and the series \eqref{eq:DefMSOCT} converges. In Section~\ref{Section:AnalyticProperties} we explore some analytical properties of $\calG_2(a,b)$ when $X\in\Mat{d}{\calA}$ is a matricial second-order semicircular element.

\begin{definition}
\label{Def:MatricialSOSemicircularElement}
We say that $X\in\Mat{d}{\calA}$ is a matricial second-order (MSO) semicircular element with covariance $\sigma:[d]^2\times[d]^2\to\R$ if
\begin{itemize}
	\item[i)] $X(p,q)=X(q,p)$ for all $p,q\in[d]$;
	\item[ii)] $\{X(p,q) \mid p,q\in[d]\}$ is a second-order semicircular family with covariance $\sigma$ (Def.~\ref{Def:SOSemicircularFamily}).
\end{itemize}
\end{definition}

Observe that by Part i), for all $p,q,r,s\in[d]$,
\eqn{eq:saCovarianceMapping}{\sigma(p,q;r,s)=\sigma(p,q;s,r)=\sigma(q,p;r,s)=\sigma(q,p;s,r).}
The next proposition establishes a moment-cumulant-like formula for MSO semicircular elements. Recall that for $i':[m]\to[d]$ and $i'':[n]\to[d]$, we let $i:[m+n]\to[d]$ be given by $i(s)=i'_s\I{s\leq m}+i''_{s-m}\I{s>m}$. Also, recall that \ndsty{\sum\limits_{i,j}} denotes the sum \ndsty{\sum\limits_{\repdc{i'}{1}{m}}\sum\limits_{\repdc{j'}{1}{m}}\sum\limits_{\repdc{i''}{1}{n}}\sum\limits_{\repdc{j''}{1}{n}}} where the indices run from $1$ to $d$.

\begin{proposition}
\label{Proposition:MomentCumulantFormula}
Let $X\in\Mat{d}{\calA}$ be a MSO semicircular element with covariance $\sigma$. Then
\eq{M_{m,n}({\bf a};{\bf b}) = \SUM{\pi\in\NC_2(m,n)} \kappa_\pi({\bf a};{\bf b})} where
\eqn{eq:DefSOCumulants}{\kappa_\pi({\bf a};{\bf b})(p,q;r,s) = \SUM{i,j} \left[\prod_{w=0}^m a_w(j'_w,i'_{w+1})\right] \left[\prod_{w=0}^n b_{n \minus w}^\T (j''_w,i''_{w+1})\right] \prod_{\underset{\pi}{u \sim v}} \sigma(i_u,j_u;i_v,j_v)}
with $j'_0:= p$, $i'_{m+1}:= q$, $j''_0:= s$, and $i''_{n+1}:= r$.
\end{proposition}

The apparently unnecessary transpose in \eqref{eq:DefSOCumulants} will simplify the expression for the double-line Cauchy transform introduced in Section~\ref{Subsection:MCCTDLP}. Note that $M_{m,n}$ and $\kappa_\pi$ depend on $X$ only through $\sigma$.

\begin{proof}[Proof of Proposition~\ref{Proposition:MomentCumulantFormula}]
By Part i) in Def.~\ref{Def:MatricialSOSemicircularElement}, we have that $X=X^\T$. In particular, we obtain that $(b_0X\cdots Xb_n)^\T=b_n^\T X \cdots Xb_0^\T$ and hence
\eqn{eq:Momentmn}{M_{m,n}({\bf a};{\bf b})(p,q;r,s) = \EOP_2(a_0X\cdots Xa_m \otimes b_n^\T X \cdots Xb_0^\T)(p,q;s,r).}
A straightforward computation leads to
\aln{
\nonumber (a_0X\cdots Xa_m)(p,q) &= \SUM{\repdc{i'}{1}{m}}\SUM{\repdc{j'}{1}{m}} a_0(j'_0,i'_1) X(i'_1,j'_1) \cdots X(i'_m,j'_m) a_m(j'_m,i'_{m+1})\\
\label{eq:MomentmnI} &= \SUM{\repdc{i'}{1}{m}}\SUM{\repdc{j'}{1}{m}} \left[\prod_{w=0}^m a_w(j'_w,i'_{w+1})\right] \prod_{w=1}^m X(i'_w,j'_w),
}
with $j'_0 = p$ and $i'_{m+1} = q$. Similarly, we have that
\eqn{eq:MomentmnII}{(b_n^\T X \cdots Xb_0^\T)(s,r)  = \SUM{\repdc{i''}{1}{n}}\SUM{\repdc{j''}{1}{n}} \left[\prod_{w=0}^n b_{n \minus w}^\T(j''_w,i''_{w+1})\right] \prod_{w=1}^n X(i''_w,j''_w),}
with $j''_0=s$ and $i''_{n+1}=r$. Since $\FI_2$ is bilinear, equations \eqref{eq:Momentmn}, \eqref{eq:MomentmnI}, and \eqref{eq:MomentmnII} imply that
\eq{M_{m,n}({\bf a};{\bf b})(p,q;r,s) = \SUM{i,j} \left[\prod_{w=0}^m a_w(j'_w,i'_{w+1})\right] \hspace{-1pt} \left[\prod_{w=0}^n b_{n \minus w}^\T(j''_w,i''_{w+1})\right] \hspace{-1pt} \FI_2\left(\prod_{w=1}^m X(i'_w,j'_w),\prod_{w=1}^n X(i''_w,j''_w)\right).}
By Part ii) in Def.~\ref{Def:MatricialSOSemicircularElement}, the entries of $X$ form a second-order semicircular family. By \eqref{eq:SOMomentsSOSF}, we have that
\eq{M_{m,n}({\bf a};{\bf b})(p,q;r,s) = \SUM{\pi\in\NC_2(m,n)} \SUM{i,j}  \left[\prod_{w=0}^m a_w(j'_w,i'_{w+1})\right] \left[\prod_{w=0}^n b_{n \minus w}^\T (j''_w,i''_{w+1})\right] \prod_{\underset{\pi}{u \sim v}} \sigma(i_u,j_u;i_v,j_v),}
as required.
\end{proof}

By abuse of notation, we let $\kappa_\pi(a;b)$ denote $\kappa_\pi(a,\ldots,a;b,\ldots,b)$ for every $a,b\in\MAT_d$ .

\begin{corollary}
\label{Corollary:MatricialCauchyTransform}
Let $X\in\Mat{d}{\calA}$ be a MSO semicircular element with covariance $\sigma$. Then its MSO Cauchy transform is given by
\eq{\calG_2(a,b) = \sum_{m,n\geq1} \sum_{\pi\in\NC_2(m,n)} \kappa_\pi(a^{-1};b^{-1}).}
\end{corollary}

In Section~\ref{Section:CauchyTransforms} we establish a recursive description for $\kappa_\pi$ which allows us to obtain an explicit expression for $\calG_2$ in terms of $\calG$. In the following section we connect the matricial second-order objects introduced here with the corresponding second-order objects associated to block random matrices.

\section{Fluctuations of block random matrices}
\label{Section:FluctuationsBlockRandomMatrices}

We start this section connecting $\EOP_2:\Mat{d}{\calA}\otimes\Mat{d}{\calA}\to\MAT_d\otimes\MAT_d$ as introduced in the previous section with the covariance of traces of block random matrices. Recall the notion of convergence in second-order distribution from Def.~\ref{Def:ConvergenceSODistribution}.

\begin{proposition}
\label{Prop:FluctuationsBlockMatrices}
For each $N\in\N$, let $\{X_N^{(p,q)},Y_N^{(p,q)} | p,q\in[d]\}$  be $N\by N$ random matrices. Consider the $dN\by dN$ block random matrices $X_N=(X_N^{(p,q)})_{p,q}$ and $Y_N=(Y_N^{(p,q)})_{p,q}$. Assume that
\eq{\{X_N^{(p,q)},Y_N^{(p,q)} | p,q\in[d]\}\stackrel{\textnormal{so-dist}}{\longrightarrow} \{X(p,q),Y(p,q) | p,q\in[d]\}}
in some second-order probability space $(\calA,\FI,\FI_2)$ where $X,Y\in\Mat{d}{\calA}$. Then
\eq{\lim_{N\to\infty} k_2\left(\Tr{X_N},\Tr{Y_N}\right) = \Tr{\EOP_2(X \otimes Y)}.}
\end{proposition}

\begin{proof}
For each $N\in\N$, let $A_N=(\id\otimes\lTr)(X_N)$ and $B_N=(\id\otimes\lTr)(Y_N)$. For all $N\in\N$, both $A_N$ and $B_N$ are $d\by d$ matrices with $A_N(p,q) = \lTr(X_N^{(p,q)})$ and $B_N(p,q) = \lTr(Y_N^{(p,q)})$ for all $p,q\in[d]$. Since the trace of the tensor product of two matrices equals the product of the traces of the individual matrices, we have that
\aln{
\nonumber k_2(\Tr{X_N},\Tr{Y_N}) &= k_2(\Tr{A_N},\Tr{B_N})\\
\nonumber &= \E{\Tr{A_N}\Tr{B_N}}-\E{\Tr{A_N}}\E{\Tr{B_N}}\\
\label{eq:CovarianceTensorProduct} &= \Tr{\E{A_N\otimes B_N}-\E{A_N}\otimes \E{B_N}}.
}
Let $C_N=\E{A_N\otimes B_N}-\E{A_N}\otimes \E{B_N}$. Observe that, for all $p,q,r,s\in[d]$,
\al{
C_N(p,q;r,s) &= \E{\Tr{X_N^{(p,q)}}\Tr{Y_N^{(r,s)}}}-\E{\Tr{X_N^{(p,q)}}}\E{\Tr{Y_N^{(r,s)}}}\\
&= k_2\left(\Tr{X_N^{(p,q)}},\Tr{Y_N^{(r,s)}}\right).
}
By the assumed convergence in second-order distribution, the previous equation implies that
\eq{\lim_{N\to\infty} C_N(p,q;r,s) = \FI_2(X(p,q),Y(r,s)) = \EOP_2(X \otimes Y)(p,q;r,s).}
Plugging the previous limit in \eqref{eq:CovarianceTensorProduct}, we obtain that \dsty{\lim_{N\to\infty} k_2\left(\Tr{X_N},\Tr{Y_N}\right) = \Tr{\EOP_2(X \otimes Y)}}.
\end{proof}

The following proposition relates the second-order Cauchy transform and its matricial counterpart. It is important to observe the similarity between the formulas in \eqref{eq:GTrG} and \eqref{eq:G2TrG2}.

\begin{proposition}
\label{Prop:MatricialSOMoments}
Let $X_N=(X_N^{(p,q)})_{p,q}$ be a $dN\by dN$ block random matrix such that
\eq{\{X_N^{(p,q)} : p,q\in[d]\} \stackrel{\textnormal{so-dist}}{\longrightarrow} \{X(p,q) : p,q\in[d]\}}
in some second-order probability space $(\calA,\FI,\FI_2)$ where $X\in\Mat{d}{\calA}$. For $m,n\in\N$, ${\bf a}\in\MAT_d^{m+1}$, and ${\bf b}\in\MAT_d^{n+1}$, let
\eq{A_N = (a_0\otimes{\rm I}_N)X_N\cdots X_N (a_m\otimes{\rm I}_N) \quad \textnormal{ and } \quad B_N = (b_0\otimes{\rm I}_N)X_N\cdots X_N(b_n\otimes{\rm I}_N).}
Similarly, let $A=a_0 X \cdots X a_m$ and $B=b_0 X \cdots X b_n$. Then,
\eq{\lim_{N\to\infty} k_2\left(\Tr{A_N},\Tr{B_N)}\right) = \Tr{\EOP_2(A \otimes B)}.}
In particular, we have that
\eqn{eq:G2TrG2}{G_2(z,w) = \Tr{\calG_2(z{\rm I}_d,w{\rm I}_d)}.}
\end{proposition}

\begin{proof}
Using the convention adopted in \eqref{eq:ConventionBlockMatrix}, we have that
\al{
A_N^{(p,q)} &= \SUM{\repdc{i'}{1}{m}}\SUM{\repdc{j'}{1}{m}} (a_0\otimes{\rm I}_N)^{(p,i'_1)} X_N^{(i'_1,j'_1)} \cdots X_N^{(i'_n,j'_n)} (a_m\otimes{\rm I}_N)^{(j'_n,q)}\\
&= \SUM{\repdc{i'}{1}{m}}\SUM{\repdc{j'}{1}{m}} (a_0(p,i'_1){\rm I}_N) X_N^{(i'_1,j'_1)} \cdots X_N^{(i'_n,j'_n)} (a_m(j'_n,q){\rm I}_N)\\
&= \SUM{\repdc{i'}{1}{m}}\SUM{\repdc{j'}{1}{m}} \left[\prod_{w=0}^m a_w(j'_w,i'_{w\plus1})\right]  \prod_{w=1}^m X_N^{(i'_w,j'_w)},
}
where $j'_0 := p$ and $i'_{n+1} := q$. Similarly,
\eq{B_N^{(r,s)} = \SUM{\repdc{i''}{1}{n}}\SUM{\repdc{j''}{1}{n}} \left[\prod_{w=0}^n b_w(j''_w,i''_{w\plus1})\right] \prod_{w=1}^n X_N^{(i''_w,j''_w)},}
where $j''_0 := r$ and $i''_{n+1} := s$. A straightforward computation shows that
\al{
A(p,q) &= \SUM{\repdc{i'}{1}{m}}\SUM{\repdc{j'}{1}{m}} \left[\prod_{w=0}^m a_w(j'_w,i'_{w\plus1})\right]  \prod_{w=1}^m X(i'_w,j'_w),\\
B(r,s) &= \SUM{\repdc{i''}{1}{n}}\SUM{\repdc{j''}{1}{n}} \left[\prod_{w=0}^n b_w(j''_w,i''_{w\plus1})\right] \prod_{w=1}^n X(i''_w,j''_w),
}
where $j'_0 = p$, $i'_{n+1} = q$, $j''_0 = r$, and $i''_{n+1} = s$. The assumed convergence in second-order distribution immediately implies that $\{A_N^{(p,q)},B_N^{(p,q)} | p,q\in[d]\}$ converges in second-order distribution to $\{A(p,q),B(p,q) | p,q\in[d]\}$. By the previous proposition, 
\eq{\lim_{N\to\infty} k_2\left(\Tr{A_N},\Tr{B_N}\right) = \Tr{\EOP_2(A \otimes B)}.}
In particular, we have that \dsty{\alpha_{m,n} := \lim_{N\to\infty} k_2(\Tr{X_N^m},\Tr{X_N^n}) = \Tr{\EOP_2(X^m \otimes X^n)}} and
\al{
G_2(z,w) &= \sum_{m,n\geq1} z^{-(m+1)} w^{-(n+1)} \alpha_{m,n}\\
&= \sum_{m,n\geq1} \Tr{\EOP_2(z^{-1}X \cdots Xz^{-1} \otimes w^{-1}X \cdots Xw^{-1})}\\
&= \Tr{\calG_2(z{\rm I}_d,w{\rm I}_d)},
}
as required.
\end{proof}

The following is a restatement of Theorem 3.1 \cite{MingoSpeicher2006} in our notation.

\begin{theorem}
Let $X_N=(X_N^{(p,q)})_{p,q}$ be block Gaussian matrix (Def.~\ref{Def:BlockGaussianMatrices}) with covariance $\sigma:[d]^2\times[d]^2\to\R$. If $X\in\Mat{d}{\calA}$ is a MSO semicircular element with covariance $\sigma$, then
\eq{\{X_N^{(p,q)} : p,q\in[d]\} \stackrel{\textnormal{so-dist}}{\longrightarrow} \{X(p,q) : p,q\in[d]\}.}
\end{theorem}

Hence, by the previous theorem and \eqref{eq:G2TrG2}, the second-order Cauchy transform of a block Gaussian matrix can be obtained from the MSO Cauchy transform of the limiting MSO semicircular element. In what follows we focus on computing the latter.

\section{Single-line, double-line and annular pairings}
\label{Section:Pairings}

In this section we consider three types of pairings: single-line, double-line, and annular. These pairings, and their relations, play an important role in the computation of the MSO Cauchy transform of MSO semicircular elements. The combinatorial facet of free probability theory centers in the notion of non-crossing partitions. Whenever we talk about partitions (pairings), we implicitly refer to non-crossing partitions (pairings).

\subsection{Single-line pairings}

Let $\NC_2(n)$ be the set of non-crossing pairings of $n$ points. In the following section we introduce a type of pairing called double-line. To make a clear distinction between the different types of pairings, we refer to the usual non-crossing pairings as {\it single-line pairings}. For uniformity of notation, we let $\SP_n:=\NC_2(n)$ for all $n\geq1$ and $\SP_0=\{\emptyset\}$. The set of all single-line pairings $\SP$ is then given by
\eq{\SP = \bigcup_{n\geq0} \SP_n.}

A word about the name single-line pairing is in order. The nesting of the blocks of a given partition is not particularly important when computing its corresponding scalar-valued free cumulant. However, this nesting is crucial in the operator-valued setting. This feature is well known in the literature \cite{Speicher1998}, but for concreteness consider the following.

\begin{example}
For a semicircular element $x\in\calA$, its scalar-valued free cumulants satisfy
\eq{\kappa_{\pi_1}^x = \F{x^2}^n = \kappa_{\pi_2}^x}
for all $\pi_1,\pi_2\in\SP_{2n}$. Let $\pi_1=\{\{1,2\},\{3,4\}\}$ and $\pi_2 = \{\{1,4\},\{2,3\}\}$. For an $\MAT_d$-valued semicircular element $X$ with covariance $\sigma$, its $\MAT_d$-valued free cumulants satisfy
\eq{\kappa_{\pi_1}^X(a,a,a,a,a) = a\eta(a)a\eta(a)a \quad \textnormal{ and } \quad \kappa_{\pi_2}^X(a,a,a,a,a) = a\eta(a\eta(a)a)a}
for all $a\in\MAT_d$ where $\eta:\MAT_d\to\MAT_d$ is given by \ndsty{\eta(a)(p,q) = \sum_{k,l} a(k,l) \sigma(p,k;l,q)}. In general, we have that $\kappa_{\pi_1}^X(a,a,a,a,a) \neq \kappa_{\pi_2}^X(a,a,a,a,a)$.
\end{example}

For notational convenience, in this paper we think of $\kappa_\pi$, with $\pi\in\SP_{2n}$, as a $2n+1$-linear map. In the usual convention, e.g., Section~9.1 in \cite{MingoSpeicher2017}, this cumulant is actually a $2n-1$-linear map. Note that these two conventions differ by a scalar matrix $a$ at the beginning and the end of the cumulant, as seen in the previous example.

In scalar-valued free probability theory it is customary to represent the elements of $\NC(n)$ as non-crossing partitions of $n$ points in the circle. In this graphical representation the nested structure of non-crossing partitions is somehow buried. As recalled in the previous example, this nesting is crucial in the operator-valued setting. To emphasize this feature, we represent the elements of $\SP_n$ as (non-crossing) pairings of $n$ points located along a single line, as depicted in Figure~\ref{Fig:SingleLinePairing}.

\begin{figure}[ht]
\centering
\includegraphics[scale=0.4]{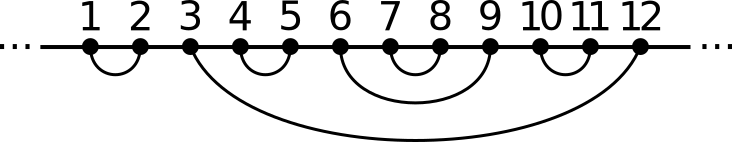}
\caption{Graphical representation of a single-line pairing.}
\label{Fig:SingleLinePairing}
\end{figure}

\subsection{Double-line pairings}
\label{Subsection:DoubleLinePairings}

In this section we introduce a type of pairing called double-line. As the name suggests, these pairings may be represented as pairings of points located along two parallel lines.

\begin{definition}
For $m,n\geq0$, we define the set of double-line pairings of $m$ and $n$ points $\DP_{m,n}$ as
\eq{\DP_{m,n} := \{(m,n,\pi) : \pi\in\NC_2(m+n)\}.}
When there is no risk of confusion, we write $\pi\in\DP_{m,n}$ instead of $(m,n,\pi)\in\DP_{m,n}$. For $\pi\in\DP_{m,n}$, a pair $\{u,v\}\in\pi$ is called a through string if either $u\leq m < v$ or $v\leq m < u$. Let $\DP^{(k)}_{m,n}$ be the set of double-line pairings of $m$ and $n$ points with exactly $k$ through strings. The set of all double-line pairings with exactly $k$ through-strings and the set of all double pairings are then given by
\eq{\DP^{(k)} = \bigcup_{m,n\geq0} \DP_{m,n}^{(k)} \quad \textnormal{ and } \quad \DP = \bigcup_{m,n\geq0} \DP_{m,n}.}
\end{definition}

Note that if $m_1\neq m_2$, and $m_1+n_1=m_2+n_2$, and $\pi\in\NC_2(m_1+n_1)$, then $(m_1,n_1,\pi)$ and $(m_2,n_2,\pi)$ are different elements of $\DP$. A word about the name double-line pairing is in order. We think of the elements of $\DP_{m,n}$ as non-crossing pairings of points located along two parallel lines, one with $m$ points labelled in increasing order and the other with $n$ points labelled in decreasing order. In this context, a through string is a pair that connects one line to the other.

\begin{example}
\label{Example:DoubleLinePairings}
Consider the non-crossing pairing
\eq{\pi=\{\{1,2\},\{3,12\},\{4,5\},\{6,9\},\{7,8\},\{10,11\}\}\in\NC_2(12).}
A graphical representation of $\pi$ is provided in Figure~\ref{Fig:SingleLinePairing}. The double-line pairings $\pi_1=(4,8,\pi)\in\DP_{4,8}$ and $\pi_2=(6,6,\pi)\in\DP_{6,6}$ are depicted in Figure~\ref{Fig:DoubleLinePairing}.
\end{example}

\begin{figure}[ht]
\centering
\includegraphics[scale=0.4]{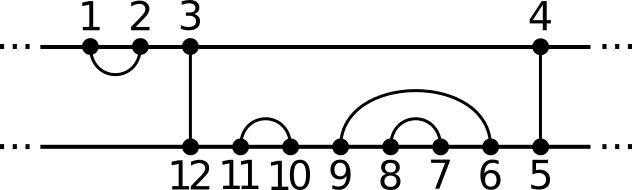} \quad \quad \quad \includegraphics[scale=0.4]{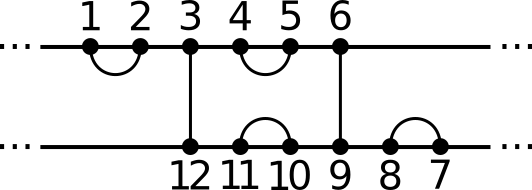}
\caption{Examples of double-line pairings.}
\label{Fig:DoubleLinePairing}
\end{figure}

The following proposition, whose proof is a routine verification, provides a useful relation between single-line and double-line pairings. Recall that for a given set $\calX$, we let $\calX^n = \calX \times \cdots \times \calX$ be its $n$-fold Cartesian product.

\begin{proposition}
\label{Prop:BijectionSPDP}
For $k\geq 0$, consider the function $\Psi^{(k)}:\SP^{k+1}\times\SP^{k+1}\to\DP^{(k)}$ given by
\al{
\Psi^{(k)}(\repdc{\pi'}{0}{k};\repdc{\pi''}{k}{0}) &= \bigcup_{i=0}^k \left(\pi'_i + i + \sum_{j=0}^{i} m_j - m_i\right)\cup \bigcup_{i=0}^k \left(\pi''_{k \minus i} + m + i + \sum_{j=0}^i n_{k \minus j} - n_{k \minus i}\right) \cup \\
& \quad \quad \quad \cup \left\{ \left\{i+\sum_{j=0}^{i-1} m_j,m+n-\left(i - 1 + \sum_{j=0}^{i-1} n_j\right) \right\} : i\in[k]\right\} \in \DP_{m,n}^{(k)},
}
where $\pi'_j\in\SP_{m_j}$ and $\pi''_j\in\SP_{n_j}$ for all $0\leq j\leq k$, $m=k+\sum_{j=0}^k m_j$, and $n=k+\sum_{j=0}^k n_j$. The mapping $\Psi^{(k)}$ is a bijection between $\SP^{k+1}\times\SP^{k+1}$ and $\DP^{(k)}$.
\end{proposition}

Naturally, there exists a bijection between $\SP^{k+1}\times\SP^{k+1}$ and $\DP^{(k)}$ as both sets are countably infinite. The importance of the previous proposition comes from the fact that the mapping $\Psi^{(k)}$ allows us to express the cumulants of double-line pairings in terms of the cumulants of single-line pairings, as shown in Section~\ref{Subsection:MCCTDLP}.

In graphical terms, $\Psi^{(k)}$ glues the single-line pairings $\repdc{\pi'}{0}{k},\repdc{\pi''}{k}{0}$ successively and intertwines them with through strings, see Figure~\ref{Fig:BijectionSPDP}. Note that $\pi'_i$ and $\pi''_i$ are in front of each other. E.g., if $\pi_1$ and $\pi_2$ are the double-line pairings in Example~\ref{Example:DoubleLinePairings}, then
\al{
\pi_1 &= \Psi^{(2)}(\{\{1,2\}\},\emptyset,\emptyset;\emptyset,\{\{1,4\},\{2,3\},\{5,6\}\},\emptyset),\\
\pi_2 &= \Psi^{(2)}(\{\{1,2\}\},\{\{1,2\}\},\emptyset;\{\{1,2\}\},\{\{1,2\}\},\emptyset).
}
Observe that in the double-line representation of $\Psi^{(k)}(\repdc{\pi'}{0}{k};\repdc{\pi''}{k}{0})$ the depictions of $\repdc{\pi''}{k}{0}$ are graphically reversed.

\begin{figure}[ht]
\centering
\includegraphics[scale=0.4]{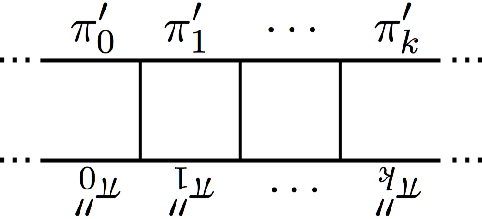}
\caption{Graphical representation of $\Psi^{(k)}(\repdc{\pi'}{0}{k};\repdc{\pi''}{k}{0})$.}
\label{Fig:BijectionSPDP}
\end{figure}

A family of double-line pairings that plays an important role in the computation of the MSO Cauchy transform of MSO semicircular elements is
\eq{\DP^{||} := \bigcup_{m,n\geq1} \left\{\pi\in\DP_{m,n} : \{1,m+n\},\{m,m+1\}\in\pi\right\}.}
Note that there is a very natural bijection between $\DP$ and $\DP^{||}\bs\DP_{1,1}$ which is implemented as follows. For a given $m,n\in\N$, let $h:\N\to\N$ be given by $h(u) = 1 + u + 2 \cdot \I{u>m}$. Then, the mapping
\eqn{eq:BijectionDPDP}{\DP_{m,n}\ni\pi \mapsto \{\{1,m+n+4\},\{m+2,m+3\}\}\cup\{\{h(u),h(v)\} : \{u,v\}\in\pi\}=\tilde{\pi}\in\DP_{m+2,n+2}}
is a bijection between $\DP_{m,n}$ and $\DP_{m+2,n+2}\cap\DP^{||}$. A graphical representation of the action of this mapping is shown in Figure~\ref{Fig:BijectionDPDP}.

\begin{figure}[ht]
\centering
\includegraphics[scale=0.4]{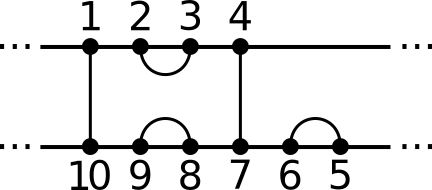} \quad \quad \includegraphics[scale=0.4]{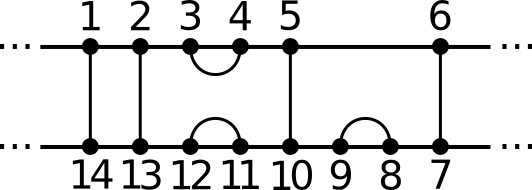}
\caption{A double-line pairing and its image under the isomorphism in \eqref{eq:BijectionDPDP}.}
\label{Fig:BijectionDPDP}
\end{figure}

\subsection{Annular Pairings}

Let $\NC_2(m,n)$ be the set of non-crossing annular pairings of $m$ exterior and $n$ interior points. For uniformity of notation, we let $\AP_{m,n}:=\NC_2(m,n)$ for all $m,n\geq1$. For $\pi\in\AP_{m,n}$, a pair $\{u,v\}\in\pi$ is called a through string if either $u\leq m < v$ or $v\leq m < u$. Let $\AP_{m,n}^{(k)}$ be the set of annular pairings of $m$ exterior and $n$ interior points with exactly $k$ through strings. The set of all annular pairings with exactly $k$ through strings and the set of all annular pairings are then given by
\eq{\AP^{(k)} = \bigcup_{m,n\geq 1} \AP_{m,n}^{(k)} \quad \textnormal{ and } \quad \AP = \bigcup_{m,n\geq 1} \AP_{m,n}.}
Recall that, by definition, a pairing in $\AP_{m,n}$ must have at least one through string. Aiming for a simpler computation of the MSO Cauchy transform of MSO semicircular elements, we divide the annular pairings into two classes.

\begin{definition}
Let $\pi\in\AP_{m,n}^{(k)}$. Let $1\leq t'_1<t'_2<\cdots<t'_k\leq m$ and $1\leq t''_1,\ldots,t''_k\leq n$ be such that $\{t'_1,m+t''_1\},\ldots,\{t'_k,m+t''_k\}$ are the through strings of $\pi$. We say that $\pi$ is of Type I, denoted by $\pi\in\calT_\textnormal{I}$, if $t''_1 \geq t''_k$. If $t''_1<t''_k$, we say that $\pi$ is of Type II, and denote it by $\pi\in\calT_\textnormal{II}$.
\end{definition}

As it is customary, \cite{MingoNica2004,MingoSpeicher2006}, we represent the elements of $\AP_{m,n}$ as non-crossing pairings of points located along two concentric circles, see Figure~\ref{Fig:ExAnnularPairing}. In this work we make special emphasis on the starting and ending points of the circles. This distinction, which is immaterial in the scalar-valued setting, is crucial in the matrix-valued computations of the following section. Graphically, Type I annular pairings have their exterior and interior circles aligned with respect to where they start and end, while Type II annular pairings are somehow shifted, see Figure~\ref{Fig:AnnularPairingTypes}.

\begin{figure}[ht]
\centering
\includegraphics[scale=0.3]{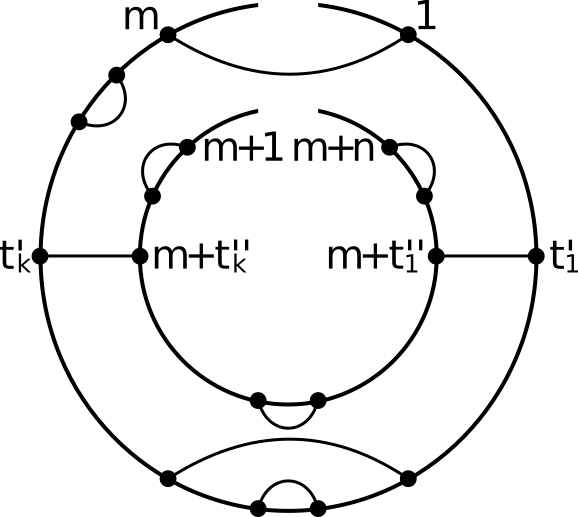} \quad \quad \quad \quad \quad \quad \quad \quad \includegraphics[scale=0.3]{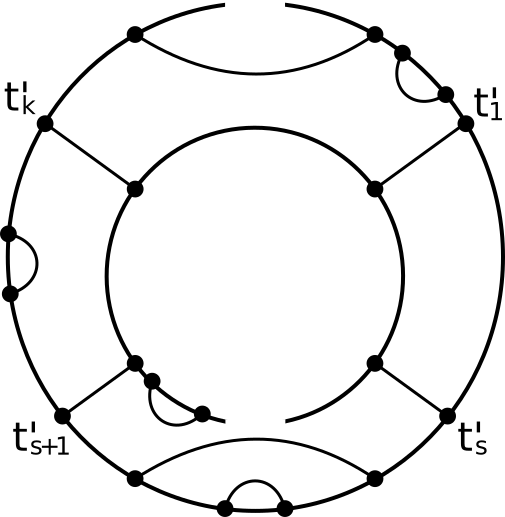}
\caption{Example of a Type I and a Type II annular pairing.}
\label{Fig:AnnularPairingTypes}
\end{figure}

\subsubsection{Type I annular pairings}

Type I annular pairings can be decomposed into three double-line pairings as follows.

\begin{definition}
Let $\pi\in\AP_{m,n}^{(k)}\cap\calT_\textnormal{I}$ and $\{t'_1,m+t''_1\},\ldots,\{t'_k,m+t''_k\}$ be its through strings with $t'_1<\cdots<t'_k$ (so $t''_1>\cdots>t''_k$). We define
\begin{itemize}
	\item[a)] the exterior boundary of $\pi$, $\pi_{eb} := \pi|_{[1,t'_1)\cup(t'_k,m]}$;
	\item[b)] the interior boundary of $\pi$, $\pi_{ib}:= \pi|_{[m+1,m+t''_k)\cup(m+t''_1,m+n]}$;
	\item[c)] the regular part of $\pi$, $\pi_{re}:= \pi|_{[t'_1,t'_k]\cup[m+t''_k,m+t''_1]}$.
\end{itemize}
\end{definition}

In Figure~\ref{Fig:TypeIPairing} there is a graphical representation of the above decomposition for the Type I annular pairing in Figure~\ref{Fig:AnnularPairingTypes}. Observe that we can identify $\pi_{eb}$ with an element in $\DP_{t'_1-1,m-t'_k}$. Specifically, we can identify $\pi_{eb}$ with
\eqn{eq:ebImplementation}{\{\{h(u),h(v)\} : \{u,v\}\in\pi_{eb}\}\in\DP_{t'_1-1,m-t'_k},}
where $h:\N\to\N$ is given by $h(u) = u \I{u<t'_1} + (u-t'_k+t'_1-1)\I{u \geq t'_1}$. Similarly, we can identify $\pi_{ib}$ with an element in $\DP_{t''_k-1,n-t''_1}$. Also, we can identify $\pi_{re}$ with the element of $\DP^{||}$ given by
\eqn{eq:reImplementation}{\{ \{h(u),h(v)\} : \{u,v\}\in\pi_{re}\}\in\DP_{t'_k-t'_1+1,t''_1-t''_k+1},}
where $h:\N\to\N$ is given by $h(u)=(u-t'_1+1)\I{u\leq m} + (u-m-t''_k+1+t'_k-t'_1+1)\I{u>m}$. For ease of notation, we denote by $\pi_{eb}$ both $\pi_{eb}$ and its associated element in $\DP$; similarly for $\pi_{ib}$ and $\pi_{re}$. The identifications implemented in \eqref{eq:ebImplementation} and \eqref{eq:reImplementation} readily imply the following.

\begin{proposition}
\label{Prop:TIDPSP}
The mapping
\eqn{eq:BijectionTIDP}{\calT_\textnormal{I}\ni\pi \mapsto (\pi_{eb},\pi_{ib},\pi_{re})\in\DP\times\DP\times\DP^{||}}
is a bijection between $\calT_{\textnormal{I}}$ and $\DP\times\DP\times\DP^{||}$.
\end{proposition}

As in Proposition~\ref{Prop:BijectionSPDP}, the existence of a bijection between $\calT_{\textnormal{I}}$ and $\DP\times\DP\times\DP^{||}$ is obvious. The relevance of the previous proposition comes from the fact that the bijection in \eqref{eq:BijectionTIDP} allows us to express the cumulants of Type I annular pairings in terms of the cumulants of double-line pairings, as shown in Section~\ref{Subsection:CumulantsAnnularPairings}.

\begin{figure}[ht]
\centering
\includegraphics[scale=0.4]{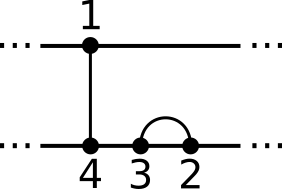} \quad \quad \quad \includegraphics[scale=0.4]{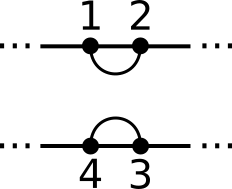} \quad \quad \quad \includegraphics[scale=0.4]{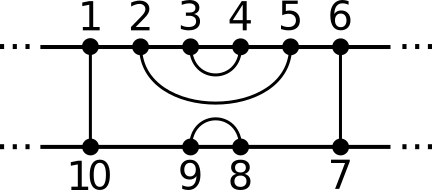}
\caption{Decomposition of a Type I annular pairing: $\pi_{eb}$, $\pi_{ib}$, and $\pi_{re}$.}
\label{Fig:TypeIPairing}
\end{figure}

\subsubsection{Type II annular pairings}

As we did with Type I annular pairings, we decompose Type II annular pairings into four double-line pairings and two single-line pairings. Even though this decomposition may seem complicated, it is very useful to compute the cumulants of Type II annular pairings, as shown in the following section.

\begin{definition}
Let $\pi\in\AP_{m,n}^{(k)}\cap\calT_\textnormal{II}$ and $\{t'_1,m+t''_1\},\ldots,\{t'_k,m+t''_k\}$ be its through strings with $t'_1<\cdots<t'_k$and $t''_{s+1}>\cdots>t''_k>t''_1>\cdots>t''_s$ for some $s\in[k]$. We define
\begin{itemize}
	\item[a)] the exterior boundary of $\pi$, $\pi_{eb} := \pi|_{[1,t'_1)\cup(t'_k,m]}$;
	\item[b)] the interior boundary of $\pi$, $\pi_{ib}:=\pi|_{[m+1,m+t''_s)\cup(m+t''_{s+1},m+n]}$;
	\item[c)] the right regular part of $\pi$, $\pi_{rr}:=\pi|_{[t'_1,t'_s]\cup[m+t''_s,m+t''_1]}$;
	\item[d)] the left regular part of $\pi$, $\pi_{lr}:=\pi|_{[t'_{s+1},t'_k]\cup[m+t''_k,m+t''_{s+1}]}$;
	\item[e)] the opposite part to the interior boundary of $\pi$, $\pi_{oi}:=\pi|_{(t'_s,t'_{s+1})}$;
	\item[f)] the opposite part to the exterior boundary of $\pi$, $\pi_{oe}:=\pi|_{(m+t''_1,m+t''_k)}$.
\end{itemize}
\end{definition}

In Figure~\ref{Fig:TypeIIPairing} there is a graphical representation of the above decomposition for the Type II annular pairing in Figure~\ref{Fig:AnnularPairingTypes}. As we did before, we can identify each of the six pieces of the previous decomposition with either a single-line or a double-line pairing. In this case, these identifications lead to the following.

\begin{proposition}
\label{Prop:TIIDPSP}
The mapping
\eq{\calT_\textnormal{II}\ni\pi \mapsto (\pi_{eb},\pi_{ib},\pi_{rr},\pi_{lr},\pi_{oi},\pi_{oe})\in\DP\times\DP\times\DP^{||}\times\DP^{||}\times\SP\times\SP}
is a bijection between $\calT_\textnormal{II}$ and $\DP\times\DP\times\DP^{||}\times\DP^{||}\times\SP\times\SP$.
\end{proposition}

\begin{figure}[ht]
\centering
\includegraphics[scale=0.4]{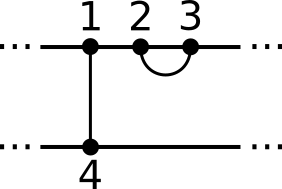} \quad \quad \includegraphics[scale=0.4]{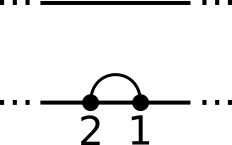} \quad \quad \includegraphics[scale=0.4]{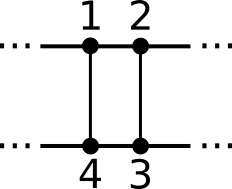} \quad \quad \includegraphics[scale=0.4]{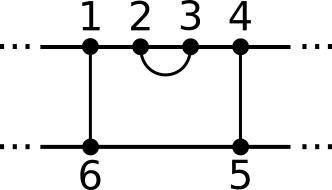}

\vspace{5pt}

{\bf -----------------------------------------------------------------------------------}

\vspace{5pt}

\includegraphics[scale=0.4]{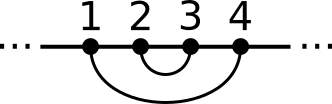} \quad \quad \quad \includegraphics[scale=0.4]{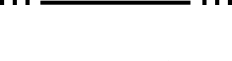}
\caption{Decomposition of a Type II annular pairing: $\pi_{eb}$, $\pi_{ib}$, $\pi_{rr}$, $\pi_{lr}$, $\pi_{oi}$, and $\pi_{oe}$.}
\label{Fig:TypeIIPairing}
\end{figure}

\section{Cauchy transforms of MSO semicircular elements}
\label{Section:CauchyTransforms}

In this section we define and compute the MSO Cauchy transforms associated to single-line, double-line, and annular pairings.

\subsection{Single-line pairings}

Recall that the single-line pairings are the usual non-crossing pairings, i.e., $\SP_n = \NC_2(n)$ for $n\geq1$ and $\SP_0=\{\emptyset\}$. We adopt the operator-valued free cumulants as the cumulants associated to single-line pairings. Specifically, if $X\in\Mat{d}{\calA}$ is a MSO semicircular element with covariance $\sigma$, the cumulant $\kappa_\pi:\MAT_d^{n+1}\to\MAT_d$ associated to $\pi\in\SP_n$ is given by
\eq{\kappa_\pi(\repdc{a}{0}{n})(p,q) = \SUM{\repdc{i}{1}{n}}\SUM{\repdc{j}{1}{n}} \left[\prod_{w=0}^n a_w(j_w,i_{w+1})\right] \prod_{\underset{\pi}{u \sim v}} \sigma(i_u,j_u;i_v,j_v),}
where $j_0:= p$ and $i_{n+1}:= q$. Recall that \cite{Speicher1998}, for all $(\repdc{a}{0}{n})\in\MAT_d^{n+1}$,
\eqn{eq:OpVMomentCumulantFormula}{\Eop{a_0X \cdots Xa_n} = \sum_{\pi\in\SP_n} \kappa_\pi(\repdc{a}{0}{n}).}
In particular, the cumulants associated to single-line pairings satisfy a moment-cumulant formula.

Single-line cumulants obey a recursive computation rule that depends on the nested structure of the underlying pairings \cite{Speicher1998}. Even though we won't explicitly rely on this property, for completeness consider the following.

\begin{example}
Let $\pi\in\SP_{12}$ be the non-crossing pairing given in Figure~\ref{Fig:SingleLinePairing}. In this case, for ${\bf a}\in\MAT_d^{13}$,
\eq{\kappa_\pi({\bf a}) = a_0 \eta(a_1) a_2 \eta(a_3 \eta(a_4) a_5 \eta(a_6 \eta(a_7) a_8) a_9 \eta(a_{10}) a_{11})a_{12},}
where $\eta:\MAT_d\to\MAT_d$ is given by \ndsty{\eta(a)(p,q) = \sum_{k,l} a(k,l) \sigma(p,k;l,q)}.
\end{example}

The following lemma will be useful later in this section.

\begin{lemma}
\label{Lemma:ReversedSPCumulant}
Let $X\in\Mat{d}{\calA}$ be a MSO semicircular element with covariance $\sigma$. If $\pi\in\SP_n$, then
\eq{\kappa_\pi(\repdc{a^\T}{n}{0})^\T = \kappa_{\hat{\pi}}(\repdc{a}{0}{n})}
for all $(\repdc{a}{0}{n})\in\MAT_d^{n+1}$ where $\hat{\pi} = \{\{n+1-u,n+1-v\} : \{u,v\}\in\pi\}$.
\end{lemma}

\begin{proof}
Let $p,q\in[d]$. By definition of $\kappa_\pi$, we have that
\al{
\kappa_\pi(\repdc{a^\T}{n}{0})(q,p) &= \SUM{\repdc{i}{1}{n}}\SUM{\repdc{j}{1}{n}} \left[\prod_{w=0}^n a_{n \minus w}^\T(j_w,i_{w+1})\right] \prod_{\underset{\pi}{u \sim v}} \sigma(i_u,j_u;i_v,j_v)\\
&= \SUM{\repdc{i}{1}{n}}\SUM{\repdc{j}{1}{n}} \left[\prod_{w=0}^n a_{n \minus w}(i_{w+1},j_w)\right] \prod_{\underset{\pi}{u \sim v}} \sigma(i_u,j_u;i_v,j_v),
}
with $j_0:= q$ and $i_{n+1}:= p$. Making the change of variables $i_w \to j'_{n+1-w}$ and $j_w\to i'_{n+1-w}$, so $j'_0=p$ and $i'_{n+1}=q$, we obtain that
\al{
\kappa_\pi(\repdc{a^\T}{n}{0})(q,p) &= \SUM{\repdc{i'}{1}{n}}\SUM{\repdc{j'}{1}{n}} \left[\prod_{w=0}^n a_{n \minus w}(j'_{n \minus w},i'_{n+1\minus w})\right] \prod_{\underset{\pi}{u \sim v}} \sigma(j'_{n+1 \minus u},i'_{n+1 \minus u};j'_{n+1 \minus v},i'_{n+1 \minus v})\\
&= \SUM{\repdc{i'}{1}{n}}\SUM{\repdc{j'}{1}{n}} \left[\prod_{w=0}^n a_w(j'_w,i'_{w+1})\right] \prod_{\underset{\pi}{u \sim v}} \sigma(j'_{n+1 \minus u},i'_{n+1 \minus u};j'_{n+1 \minus v},i'_{n+1 \minus v}).
}
The change of variables $n+1-u \to u$ and $n+1-v \to v$ leads to
\eq{\prod_{\underset{\pi}{u \sim v}} \sigma(j'_{n+1 \minus u},i'_{n+1 \minus u};j'_{n+1 \minus v},i'_{n+1 \minus v}) = \prod_{\underset{\pi}{n + 1 \minus u \sim n + 1 \minus v}} \sigma(j'_u,i'_u;j'_v,i'_v).}
By definition of $\hat{\pi}$, we have that $\{n+1-u,n+1-v\}\in\pi$ if and only if $\{u,v\}\in\hat{\pi}$. Therefore,
\eq{\kappa_\pi(\repdc{a^\T}{n}{0})(q,p) = \SUM{\repdc{i'}{1}{n}}\SUM{\repdc{j'}{1}{n}} \left[\prod_{w=0}^n a_w(j'_w,i'_{w+1})\right] \prod_{\underset{\hat{\pi}}{u \sim v}} \sigma(j'_{u},i'_{u};j'_{v},i'_{v}).}
By definition of a MSO semicircular element, the covariance mapping $\sigma$ satisfies \eqref{eq:saCovarianceMapping}. In particular,
\eq{\kappa_\pi(\repdc{a^\T}{n}{0})(q,p) = \SUM{\repdc{i'}{1}{n}}\SUM{\repdc{j'}{1}{n}} \left[\prod_{w=0}^n a_w(j'_w,i'_{w+1})\right] \prod_{\underset{\hat{\pi}}{u \sim v}} \sigma(i'_{u},j'_{u};i'_{v},j'_{v}) = \kappa_{\hat{\pi}}(\repdc{a}{0}{n})(p,q).}
Since this equation holds for every $p,q\in[d]$, we conclude that $\kappa_\pi(\repdc{a^\T}{n}{0})^\T = \kappa_{\hat{\pi}}(\repdc{a}{0}{n})$.
\end{proof}

By abuse of notation, we let $\kappa_\pi(a):=\kappa_\pi(a,\ldots,a)$.

\begin{definition}
\label{Def:CauchyTransformSP}
Let $X\in\Mat{d}{\calA}$ be a MSO semicircular element with covariance $\sigma$. We define the single-line Cauchy transform $G_S:\MAT_d\to\MAT_d$ of $X$ by
\eq{G_S(a) = \sum_{\pi\in\SP} \kappa_\pi(a^{-1}).}
\end{definition}

The purpose of the subindex $S$ in $G_S$ is to distinguish this Cauchy transform, that depends only on the single-line pairings, from the Cauchy transforms in the following sections. Note that $G_S$ is a particular case of the operator-valued Cauchy transform in \cite{Speicher1998}. It is important to remark that $G_S$ is a formal expression version of the analytic mapping $\calG$ considered in the introduction. Since $\calG$ can be easily computed using the fixed point equation \eqref{eq:FixedPointEquationcalG}, we use the single-line Cauchy transform as the basic building block for the upcoming Cauchy transforms.

\subsection{Double-line pairings}
\label{Subsection:MCCTDLP}

In this section we define the cumulants associated to the double-line pairings and relate them to the cumulants of the single-line pairings. For a covariance mapping $\sigma:[d]^2\times [d]^2\to\R$, we let $\Sigma$ be the $d^2\times d^2$ matrix determined by $\Sigma(p,q;r,s)=\sigma(p,q;r,s)$. Given $i':[m]\to[d]$ and $i'':[n]\to[d]$, we let $i:[m+n]\to[d]$ be the function $i_p = i'_p \I{p\leq m}+i''_{p-m}\I{p>m}$. Recall that \ndsty{\sum\limits_{i,j}} denotes the sum \ndsty{\sum\limits_{\repdc{i'}{1}{m}}\sum\limits_{\repdc{j'}{1}{m}}\sum\limits_{\repdc{i''}{1}{n}}\sum\limits_{\repdc{j''}{1}{n}}} where the indices run from $1$ to $d$.

\begin{definition}
\label{Def:CumulantDP}
Let $X\in\Mat{d}{\calA}$ be a MSO semicircular element with covariance $\sigma$. We define the cumulant $\kappa_\pi:\MAT_d^{m+1}\times\MAT_d^{n+1}\to\MAT_{d^2}$ associated to $\pi\in\DP_{m,n}$ by
\eq{\kappa_\pi({\bf a};{\bf b})(p,q;r,s) = \SUM{i,j} \left[\prod_{w=0}^m a_w(j'_w,i'_{w+1})\right] \left[\prod_{w=0}^n b_{n \minus w}^\T (j''_w,i''_{w+1})\right] \prod_{\underset{\pi}{u \sim v}} \sigma(i_u,j_u;i_v,j_v),}
with $j'_0:= p$, $i'_{m+1}:= q$, $j''_0:= s$, and $i''_{n+1}:= r$.
\end{definition}

Despite the fact that these functions are called cumulants, in principle, they do not satisfy a moment-cumulant formula like the one in \eqref{eq:OpVMomentCumulantFormula}. Nonetheless, these cumulants play an important role in the computation of the MSO Cauchy transform of MSO semicircular elements. The following lemma will be useful later in this section.

\begin{lemma}
\label{Lemma:ReversedDPCumulant}
Let $X\in\Mat{d}{\calA}$ be a MSO semicircular element with covariance $\sigma$. If $\pi\in\DP_{m,n}$, then
\eq{\kappa_\pi(\repdc{a}{0}{m};\repdc{b}{0}{n})(p,q;r,s) = \kappa_{\hat{\pi}}(\repdc{b}{0}{n};\repdc{a}{0}{m})(r,s;p,q)}
where $\hat{\pi}=\{\{m+n+1-u,m+n+1-v\} : \{u,v\}\in\pi\}\in\DP_{n,m}$.
\end{lemma}

\begin{proof}
By definition, we have that
\eq{\kappa_\pi({\bf a};{\bf b})(p,q;r,s) = \SUM{i,j}  \left[\prod_{w=0}^m a_w(j'_w,i'_{w+1})\right] \left[\prod_{w=0}^n b_{n \minus w}^\T(j''_w,i''_{w+1})\right] \prod_{\underset{\pi}{u \sim v}} \sigma(i_u,j_u;i_v,j_v),}
with $j'_0:= p$, $i'_{m+1}:= q$, $j''_0:= s$, and $i''_{n+1}:= r$. Recall that, by assumption, the covariance mapping $\sigma$ satisfies \eqref{eq:saCovarianceMapping}. A straightforward computation shows that
\eq{\kappa_\pi({\bf a};{\bf b})(p,q;r,s) = \SUM{i,j} \left[\prod_{w=0}^n b_{n \minus w}(i''_{w+1},j''_w)\right] \left[\prod_{w=0}^m a_w^\T(i'_{w+1},j'_w)\right] \prod_{\underset{\pi}{u \sim v}} \sigma(j_u,i_u;j_v,i_v).}
Consider the change of variables
\eq{i''_w \to \hat{j}'_{n+1-w}\textnormal{ and } j''_w \to \hat{i}'_{n+1-w}\textnormal{ for all } w\in[n],}
\eq{i'_w \to \hat{j}''_{m+1-w}\textnormal{ and } j'_w \to \hat{i}''_{m+1-w}\textnormal{ for all } w\in[m],}
with $\hat{j}'_0:= r$, $\hat{i}'_{n+1}:= s$, $\hat{j}''_0:= q$, and $\hat{i}''_{m+1}:= p$. Observe that
\eq{i_p = i'_p \I{p\leq m} + i''_{p-m} \I{p>m} = \hat{j}''_{m+1-p} \I{n+1\leq m+n+1-p} + \hat{j}'_{m+n+1-p} \I{n+1 > m+n+1-p} = \hat{j}_{m+n+1-p}.}
Similarly, we have that $j_p=\hat{i}_{m+n+1-p}$. Thus,
\al{
\kappa_\pi({\bf a};{\bf b})(p,q;r,s) &= \SUM{\hat{i},\hat{j}} \left[\prod_{w=0}^n b_{n \minus w}(\hat{j}'_{n \minus w},\hat{i}'_{n \minus w+1})\right] \left[\prod_{w=0}^m a_w^\T(\hat{j}''_{m \minus w},\hat{i}''_{m \minus w+1})\right] \times\\
& \quad \quad \quad \quad \quad \times \prod_{\underset{\pi}{u \sim v}} \sigma(\hat{i}_{m+n+1 \minus u},\hat{j}_{m+n+1 \minus u};\hat{i}_{m+n+1 \minus v},\hat{j}_{m+n+1 \minus v}).}
Since $\{u,v\}\in\pi$ if and only if $\{m+n+1-u,m+n+1-v\}\in\hat{\pi}$, the change of variables $m+n+1-u \to u$ and $m+n+1-v \to v$ leads to
\eq{\kappa_\pi({\bf a};{\bf b})(p,q;r,s) = \SUM{\hat{i},\hat{j}} \left[\prod_{w=0}^n b_w(\hat{j}'_w,\hat{i}'_{w+1})\right] \left[\prod_{w=0}^m a_{m \minus w}^\T(\hat{j}''_w,\hat{i}''_{w+1})\right] \prod_{\underset{\hat{\pi}}{u \sim v}} \sigma(\hat{i}_{u},\hat{j}_{u};\hat{i}_{v},\hat{j}_{v}).}
Observe that the right hand side of the previous equation equals $\kappa_{\hat{\pi}}(\repdc{b}{0}{n};\repdc{a}{0}{m})(r,s;p,q)$.
\end{proof}

In the following lemma we express the cumulant associated to a double-line pairing in terms of the cumulants of single-line pairings. Let $\Psi^{(k)}:\SP^{k+1}\times\SP^{k+1}\to\DP^{(k)}$ be the bijection introduced in Proposition~\ref{Prop:BijectionSPDP}.

\begin{lemma}
\label{Lemma:CumulantDPSP}
Let $X\in\Mat{d}{\calA}$ be a MSO semicircular element with covariance $\sigma$. If $\pi\in\DP_{m,n}^{(k)}$ has through strings $\{t'_1,m+t''_1\},\ldots,\{t'_k,m+t''_k\}$ with $t'_1<\cdots<t'_k$ (so $t''_1>\cdots>t''_k$), then
\eq{\kappa_\pi({\bf a};{\bf b}) = \left[\kappa_{\pi'_0}(\repdc{a}{0}{t'_1-1}) \otimes \kappa_{\widehat{\pi''_0}}(\repdc{b}{0}{n-t''_1})\right] \Sigma \cdots \Sigma \left[\kappa_{\pi'_k}(\repdc{a}{t'_k}{m}) \otimes \kappa_{\widehat{\pi''_k}}(\repdc{b}{n-t''_k+1}{n})\right],}
where $\pi=\Psi^{(k)}(\repdc{\pi'}{0}{k};\repdc{\pi''}{k}{0})$.
\end{lemma}

To illustrate the previous lemma, consider the following.

\begin{example}
Let $\pi_1$ and $\pi_2$ be the double-line pairings in Example~\ref{Example:DoubleLinePairings}. For appropriate ${\bf a}$ and ${\bf b}$,
\al{
\kappa_{\pi_1}({\bf a};{\bf b}) &= [(a_0\eta(a_1)a_2)\otimes b_0] \Sigma [a_3\otimes(b_1\eta(b_2)b_3\eta(b_4\eta(b_5)b_6)b_7)] \Sigma [a_4 \otimes b_8],\\
\kappa_{\pi_2}({\bf a};{\bf b}) &= [(a_0\eta(a_1)a_2)\otimes b_0] \Sigma [(a_3\eta(a_4)a_5)\otimes(b_1\eta(b_2)b_3)] \Sigma [a_6\otimes(b_4\eta(b_5)b_6)].
}
\end{example}

\begin{proof}[Proof of Lemma~\ref{Lemma:CumulantDPSP}]
We proof this lemma by induction on $k\geq0$. If $\pi\in\DP_{m,n}^{(0)}$, Proposition~\ref{Prop:BijectionSPDP} implies that $\pi=\Psi^{(0)}(\pi'_0;\pi''_0)$ where $\pi'_0$ and $\pi''_0$ are pairings of $[1,m]$ and $[1,n]$, respectively. In particular,
\eq{\prod_{\underset{\pi}{u \sim v}} \sigma(i_u,j_u;i_v,j_v) = \prod_{\underset{\pi'_0}{u \sim v}} \sigma(i'_u,j'_u;i'_v,j'_v) \times \prod_{\underset{\pi''_0}{u \sim v}} \sigma(i''_u,j''_u;i''_v,j''_v)}
for every $i',j':[m]\to[d]$ and $i'',j'':[n]\to d$. Fix $p,q,r,s\in[d]$. The previous equation implies that
\al{
\kappa_\pi({\bf a};{\bf b})(p,q;r,s) &= \SUM{\repdc{i'}{1}{m}}\SUM{\repdc{j'}{1}{m}}  \left[\prod_{w=0}^m a_w(j'_w,i'_{w+1})\right] \prod_{\underset{\pi'_0}{u \sim v}} \sigma(i'_u,j'_u;i'_v,j'_v)\times\\
& \quad \quad \quad \quad \quad \quad \quad \quad \times \SUM{\repdc{i''}{1}{n}}\SUM{\repdc{j''}{1}{n}} \left[\prod_{w=0}^n b_{n \minus w}^\T(j''_w,i''_{w+1})\right] \prod_{\underset{\pi''_0}{u \sim v}} \sigma(i''_u,j''_u;i''_v,j''_v),
}
with $j'_0:= p$, $i'_{m+1}:= q$, $j''_0:= s$ and $i''_{n+1}:= r$. The right-hand side of the previous equation is the product the cumulants of two single-line pairings. Specifically,
\eq{\kappa_\pi({\bf a};{\bf b})(p,q;r,s) = \kappa_{\pi_0'}(\repdc{a}{0}{n})(p,q) \kappa_{\pi_0''}(\repdc{b^\T}{n}{0})(s,r).}
By Lemma~\ref{Lemma:ReversedSPCumulant}, we have that
\al{
\kappa_\pi({\bf a};{\bf b})(p,q;r,s) &= \kappa_{\pi_0'}(\repdc{a}{0}{n})(p,q) \kappa_{\widehat{\pi_0''}}(\repdc{b}{0}{n})(r,s) = \left[\kappa_{\pi_0'}({\bf a}) \otimes \kappa_{\widehat{\pi_0''}}({\bf b})\right](p,q;r,s).
}
Since the last equality holds for all $p,q,r,s\in[d]$, we conclude that $\kappa_\pi({\bf a},{\bf b})=\kappa_{\pi_0'}({\bf a}) \otimes \kappa_{\widehat{\pi_0''}}({\bf b})$.

Assume that the lemma is true for $k-1$. Let $\pi\in\DP_{m,n}^{(k)}$ be as in the statement of the lemma. By definition of $\Psi^{(k)}$, we have that
\eq{\pi = \pi_0' \cup (\pi_0''+m+t_1'') \cup \{t_1',m+t_1''\} \cup \tilde{\pi},}
where $\tilde{\pi}=\pi|_{(t'_1,m+t''_1)}$ is a pairing of $\{t_1'+1,\ldots,m+t_1''-1\}$. Recall that, by assumption, the covariance mapping $\sigma$ satisfies \eqref{eq:saCovarianceMapping}. The previous equality implies that \dsty{\prod_{\underset{\pi}{u \sim v}} \sigma(i_u,j_u;i_v,j_v)} equals
\al{
\prod_{\underset{\pi'_0}{u \sim v}} \sigma(i'_u,j'_u;i'_v,j'_v) \times \prod_{\underset{\pi''_0+m+t_1''}{u \sim v}} \sigma(i''_{u \minus m},j''_{u \minus m};i''_{v \minus m},j''_{v \minus m}) \times \sigma(i'_{t_1'},j'_{t_1'};j''_{t_1''},i''_{t_1''}) \times \prod_{\underset{\tilde{\pi}}{u \sim v}} \sigma(i_u,j_u;i_v,j_v),}
and therefore
\eqn{eq:InductionSplit}{\kappa_\pi({\bf a};{\bf b})(p,q;r,s) = \SUM{i'_{t_1'},j'_{t_1'},i''_{t_1''},j''_{t_1''}} A(i'_{t_1'},j''_{t_1''}) \Sigma(i'_{t_1'},j'_{t_1'};j''_{t_1''},i''_{t_1''}) B(j'_{t_1'},i''_{t_1''}),}
where
\al{
A(i'_{t_1'},j''_{t_1''}) &= \SUM{\repdc{i'}{1}{t_1' \minus 1}}\SUM{\repdc{j'}{1}{t_1' \minus 1}} \left[\prod_{w=0}^{t_1'-1} a_w(j'_w,i'_{w+1})\right] \prod_{\underset{\pi'_0}{u \sim v}} \sigma(i'_u,j'_u;i'_v,j'_v) \times\\
& \quad \quad \quad \times \SUM{\repdc{i''}{t_1''+1}{n}}\SUM{\repdc{j''}{t_1''+1}{n}} \left[\prod_{w=t_1''}^n b_{n \minus w}^\T(j''_w,i''_{w+1})\right] \prod_{\underset{\pi''_0+m+t_1''}{u \sim v}} \sigma(i''_{u \minus m},j''_{u \minus m};i''_{v \minus m},j''_{v \minus m}),\\
B(j'_{t_1'},i''_{t_1''}) &= \SUM{\repdc{i'}{t_1'+1}{m}\\\repdc{j'}{t_1'+1}{m}} \SUM{\repdc{i''}{1}{t_1'' \minus 1}\\\repdc{j''}{1}{t_1'' \minus 1}}  \left[\prod_{w=t_1'}^m a_w(j'_w,i'_{w+1})\right] \left[\prod_{w=0}^{t_1''-1} b_{n \minus w}^\T(j''_w,i''_{w+1})\right] \prod_{\underset{\tilde{\pi}}{u \sim v}} \sigma(i_u,j_u;i_v,j_v),
}
with $j'_0:= p$, $i'_{m+1}:= q$, $j''_0:= s$, and $i''_{n+1}:= r$. Note that $u \sim v$ under $\pi_0''+m+t_1''$ if and only if $u-m-t_1'' \sim v-m-t_1''$ under $\pi_0''$. The change of variables $u - m - t_1'' \to u$ and $v - m - t_1'' \to v$ leads to
\eq{\prod_{\underset{\pi''_0+m+t_1''}{u \sim v}} \sigma(i''_{u \minus m},j''_{u \minus m};i''_{v \minus m},j''_{v \minus m}) = \prod_{\underset{\pi''_0}{u \sim v}} \sigma(i''_{u + t_1''},j''_{u + t_1''};i''_{v + t_1''},j''_{v + t_1''}).}
Similarly, the change of variables $w-t_1'' \to w$ leads to
\eq{\prod_{w=t_1''}^n b_{n \minus w}^\T(j''_w,i''_{w+1}) = \prod_{w=0}^{n-t_1''} b_{n \minus t_1'' \minus w}^\T(j''_{w+t_1''},i''_{w+1+t_1''}).}
A straightforward change of variables and the base of induction then show that
\eqn{eq:InductionBase}{A(i'_{t_1'},j''_{t_1''}) = \left[\kappa_{\pi_0'}(a_0,\ldots,a_{t_1'-1}) \otimes \kappa_{\widehat{\pi_0''}}(\repdc{b}{0}{n-t_1''})\right](p,i'_{t_1'};r,j''_{t_1''}).}
Note that $\tilde{\pi}-t_1'$ is an element of $\DP_{m-t'_1,t''_1-1}^{(k-1)}$. Mutatis mutandis, changing variables we obtain that
\eqn{eq:InductionDefinitionKappa}{B(j'_{t_1'},i''_{t_1''}) = \kappa_{\tilde{\pi} \minus t_1'}(\repdc{a}{t_1'}{m};\repdc{b}{n \minus t_1''+1}{n})(j'_{t_1'},q;i''_{t_1''},s).}
Recall the product formula in \eqref{eq:SumTensorProduct}. By plugging \eqref{eq:InductionBase} and \eqref{eq:InductionDefinitionKappa} in  \eqref{eq:InductionSplit}, we conclude that
\eq{\kappa_\pi({\bf a};{\bf b}) = \left[\kappa_{\pi_0'}(a_0,\ldots,a_{t_1'-1}) \otimes \kappa_{\widehat{\pi_0''}}(\repdc{b}{0}{n-t_1''})\right] \Sigma \kappa_{\tilde{\pi} \minus t_1'}(\repdc{a}{t_1'}{m};\repdc{b}{n-t_1''+1}{n}).}
The result follows after applying the induction hypothesis to $\kappa_{\tilde{\pi} \minus t_1'}$.
\end{proof}

With the usual abuse of notation, the previous lemma reads as
\eqn{eq:CumulantDPSP}{\kappa_{\Psi^{(k)}(\repdc{\pi'}{0}{k};\repdc{\pi''}{k}{0})}(a;b) = \left[\kappa_{\pi'_0}(a) \otimes \kappa_{\widehat{\pi''_0}}(b)\right] \Sigma \cdots \Sigma \left[\kappa_{\pi'_k}(a) \otimes \kappa_{\widehat{\pi''_k}}(b)\right].}
As with $\SP$, there is a Cauchy transform associated to $\DP$.

\begin{definition}
Let $X\in\Mat{d}{\calA}$ be a MSO semicircular element with covariance $\sigma$. We define the double-line Cauchy transform $G_D:\MAT_d\times\MAT_d\to\MAT_{d^2}$ of $X$ by
\eq{G_D(a,b) = \sum_{\pi\in\DP} \kappa_\pi(a^{-1};b^{-1}).}
\end{definition}

This Cauchy transform is closely related to the single-line Cauchy transform.

\begin{proposition}
\label{Prop:CauchyTranformDPSP}
Let $X\in\Mat{d}{\calA}$ be a MSO semicircular element with covariance $\sigma$. The double-line Cauchy transform of $X$ is given by
\eq{G_D(a,b) = \left[G_S(a)\otimes G_S(b)\right] \left({\rm I}_{d^2} - \Sigma \left[G_S(a)\otimes G_S(b)\right]\right)^{-1}.}
\end{proposition}

\begin{proof}
Observe that
\eqn{eq:ExpansionGDCumulants}{G_D(a,b) = \sum_{\pi\in\DP} \kappa_\pi(a^{-1};b^{-1}) = \sum_{k\geq0} \sum_{\pi\in\DP^{(k)}} \kappa_\pi(a^{-1};b^{-1}).}
The bijection in Proposition~\ref{Prop:BijectionSPDP} implies that
\eq{\sum_{\pi\in\DP^{(k)}} \kappa_\pi(a^{-1};b^{-1}) = \SUM{\repdc{\pi'}{0}{k}\in\SP}\SUM{\repdc{\pi''}{0}{k}\in\SP} \kappa_{\Psi^{(k)}(\pi'_0,\ldots,\pi'_k,\repdc{\pi''}{k}{0})}(a^{-1};b^{-1}).}
By \eqref{eq:CumulantDPSP}, Definition~\ref{Def:CauchyTransformSP}, and the fact that $\{\hat{\pi} : \pi\in\SP\} = \SP$,
\al{
\sum_{\pi\in\DP^{(k)}} \kappa_\pi(a^{-1};b^{-1}) & = \SUM{\repdc{\pi'}{0}{k}\in\SP}\SUM{\repdc{\pi''}{0}{k}\in\SP} \left[\kappa_{\pi'_0}(a^{-1}) \otimes \kappa_{\widehat{\pi''_0}}(b^{-1})\right] \Sigma \cdots \Sigma \left[\kappa_{\pi'_k}(a^{-1}) \otimes \kappa_{\widehat{\pi''_k}}(b^{-1})\right]\\
& = \left[G_S(a)\otimes G_S(b)\right] \left(\Sigma \left[G_S(a)\otimes G_S(b)\right]\right)^k.
}
Plugging the previous equation in \eqref{eq:ExpansionGDCumulants}, we conclude that
\al{
G_D(a,b) &= \sum_{k\geq 0} \left[G_S(a)\otimes G_S(b)\right] \left(\Sigma \left[G_S(a)\otimes G_S(b)\right]\right)^k\\
&= \left[G_S(a)\otimes G_S(b)\right] \left({\rm I}_{d^2} -\Sigma \left[G_S(a)\otimes G_S(b)\right]\right)^{-1},
}
as required.
\end{proof}

The following function will play an important role in the next section.

\begin{definition}
Let $X\in\Mat{d}{\calA}$ be a MSO semicircular element with covariance $\sigma$. We define
\eq{H(a,b) = \sum_{\pi\in\DP^{||}} \kappa_\pi({\rm I}_d,a^{-1},\ldots,a^{-1},{\rm I}_d;{\rm I}_d,b^{-1},\ldots,b^{-1},{\rm I}_d).}
\end{definition}

In this case, we have the following.

\begin{proposition}
\label{Prop:H}
Let $X\in\Mat{d}{\calA}$ be a MSO semicircular element with covariance $\sigma$. Then
\eq{H(a,b) = ({\rm I}_{d^2}-\Sigma[G_S(a)\otimes G_S(b)])^{-1}\Sigma.}
\end{proposition}

\begin{proof}
Recall the bijection between $\DP$ and $\DP^{||}\bs\DP_{1,1}$ established at the end of Section~\ref{Subsection:DoubleLinePairings}, and denote it by \dsty{\DP_{m,n}\ni\pi \leftrightarrow \tilde{\pi}\in\DP_{m+2,n+2}\cap\DP^{||}}. Since $\kappa_\emptyset(a)=a$ for all $a\in\MAT_d$, equation \eqref{eq:CumulantDPSP} implies that
\eq{\kappa_{\tilde{\pi}}({\rm I}_d,a^{-1},\ldots,a^{-1},{\rm I}_d;{\rm I}_d,b^{-1},\ldots,b^{-1},{\rm I}_d) = \Sigma \kappa_\pi(a^{-1};b^{-1}) \Sigma.}
By the definition of $H(a,b)$, we obtain that
\al{
H(a,b) &= \Sigma + \sum_{\tilde{\pi}\in\DP^{||}\bs\DP_{1,1}} \kappa_{\tilde{\pi}}({\rm I}_d,a^{-1},\ldots,a^{-1},{\rm I}_d;{\rm I}_d,b^{-1},\ldots,b^{-1},{\rm I}_d)\\
&= \Sigma+\sum_{\pi\in\DP} \Sigma \kappa_\pi(a^{-1};b^{-1}) \Sigma\\
&= \Sigma + \Sigma G_D(a,b) \Sigma.
}
Proposition~\ref{Prop:CauchyTranformDPSP} then leads to $H(a,b) = \left({\rm I}_{d^2} - \Sigma \left[G_S(a)\otimes G_S(b)\right]\right)^{-1} \Sigma$.
\end{proof}

We finish this section with the following observation. Let $\pi\in\DP_{m,n}$. By Def.~\ref{Def:CumulantDP} we have that
\al{
\kappa_{\pi}({\bf a};{\bf b})(p,q;r,s) &= \SUM{\repdc{i'}{1}{m}\\\repdc{j'}{1}{m}} \SUM{\repdc{i''}{1}{n}\\\repdc{j''}{1}{n}}  \left[\prod_{w=0}^m a_w(j'_w,i'_{w+1})\right] \left[\prod_{w=0}^n b_{n \minus w}^\T(j''_w,i''_{w+1})\right] \prod_{\underset{\pi}{u \sim v}} \sigma(i_u,j_u;i_v,j_v),
}
with $j'_0:= p$, $i'_{m+1}:= q$, $j''_0:= s$, and $i''_{n+1}:= r$. If $a_0=a_m=b_0=b_n={\rm I}_d$, then
\al{
\kappa_{\pi}({\bf a};{\bf b})(p,q;r,s) &= \SUM{\repdc{i'}{2}{m}\\\repdc{j'}{1}{m\minus 1}} \SUM{\repdc{i''}{2}{n}\\\repdc{j''}{1}{n\minus 1}} \left[\prod_{w=1}^{m\minus 1} a_w(j'_w,i'_{w\plus 1})\right] \left[\prod_{w=1}^{n\minus 1} b_{n\minus w}^\T(j''_w,i''_{w\plus 1})\right] \prod_{\underset{\pi}{u \sim v}} \sigma(i_u,j_u;i_v,j_v),
}
where $i'_1:= p$, $j'_m:= q$, $i''_1:= s$, and $j''_n:= r$.

\subsection{Annular pairings}
\label{Subsection:CumulantsAnnularPairings}

Motivated by \eqref{eq:DefSOCumulants}, in this section we define the cumulants associated to the annular pairings and compute the corresponding Cauchy transform. Specifically, we have the following.

\begin{definition}
\label{Def:CumulantAP}
Let $X\in\Mat{d}{\calA}$ be a MSO semicircular element with covariance $\sigma$. We define the cumulant $\kappa_\pi:\MAT_d^{m+1}\times\MAT_d^{n+1}\to\MAT_{d^2}$ associated to $\pi\in\AP_{m,n}$ by
\eq{\kappa_\pi({\bf a};{\bf b})(p,q;r,s) = \SUM{i,j} \left[\prod_{w=0}^m a_w(j'_w,i'_{w+1})\right] \left[\prod_{w=0}^n b_{n \minus w}^\T(j''_w,i''_{w+1})\right] \prod_{\underset{\pi}{u \sim v}} \sigma(i_u,j_u;i_v,j_v),}
where $j'_0:= p$, $i'_{m+1}:= q$, $j''_0:= s$, and $i''_{n+1}:= r$.
\end{definition}

The name cumulant is justified by the fact that
\eq{\EOP_2(a_0Xa_1\cdots a_{m-1}Xa_m \otimes b_0Xb_1\cdots b_{n-1}Xb_n) = \sum_{\pi\in\AP_{m,n}} \kappa_\pi(\repdc{a}{0}{m};\repdc{b}{0}{n})}
whenever $X$ is a MSO semicircular element with covariance $\sigma$, as established in Proposition~\ref{Proposition:MomentCumulantFormula}. As we did before, we study $\AP$ via the classes $\calT_\textnormal{I}$ and $\calT_\textnormal{II}$.

\subsubsection{Type I annular pairings}

Let $\Theta:\MAT_{d^2}\to\MAT_{d^2}$ be the mapping determined by $\Theta(A)(p,q;r,s)=A(p,r;q,s)$. For example, when $d=2$, then
\eq{\Theta\left(\begin{matrix}A_{11} & A_{12} & B_{11} & B_{12} \\ A_{21} & A_{22} & B_{21} & B_{22} \\ C_{11} & C_{12} & D_{11} & D_{12} \\ C_{21} & C_{22} & D_{21} & D_{22}\end{matrix}\right) = \left(\begin{matrix}A_{11} & A_{12} & A_{21} & A_{22} \\ B_{11} & B_{12} & B_{21} & B_{22} \\ C_{11} & C_{12} & C_{21} & C_{22} \\ D_{11} & D_{12} & D_{21} & D_{22} \end{matrix}\right).}

\begin{lemma}
\label{Lemma:TypeI}
Let $X\in\Mat{d}{\calA}$ be a MSO semicircular element with covariance $\sigma$. If $\pi\in\AP_{m,n}^{(k)}\cap\calT_\textnormal{I}$ and $\{t'_1,m+t''_1\},\ldots,\{t'_k,m+t''_k\}$ are its through strings with $t'_1<\cdots<t'_k$ (so $t''_1>\cdots>t''_k$), then \dsty{\kappa_\pi({\bf a};{\bf b}) = \Theta\left(\kappa_{eb} \Theta[\kappa_{re}] \kappa_{ib}^\T\right)} where
\al{
\kappa_{re} &:= \kappa_{\pi_{re}}({\rm I}_d,\repdc{a}{t'_1}{t'_k \minus 1},{\rm I}_d;{\rm I}_d,\repdc{b}{n \minus t''_1+1}{n \minus t''_k},{\rm I}_d),\\
\kappa_{eb} &:= \kappa_{\pi_{eb}}(\repdc{a}{0}{t'_1 \minus 1};\repdc{a^\T}{m}{t'_k}),\\
\kappa_{ib} &:= \kappa_{\widehat{\pi_{ib}}}(\repdc{b}{0}{n \minus t''_1};\repdc{b^\T}{n}{n \minus t''_k+1}).
}
\end{lemma}

To illustrate the previous lemma, consider the following.

\begin{example}
\label{Example:CumulantTI}
Let $\pi_1$ be the Type I annular pairing in Figure~\ref{Fig:AnnularPairingTypes}. For ${\bf a}\in\MAT_d^{11}$ and ${\bf b}\in\MAT_d^{9}$,
\al{
\kappa_{re} &= \Sigma [(a_2\eta(a_3\eta(a_4)a_5)a_6)\otimes(b_3\eta(b_4)b_5)] \Sigma,\\
\kappa_{eb} &= [a_0\otimes a_{10}^\T] \Sigma [a_1\otimes(a_9^\T\eta(a_8^\T)a_7^\T)],\\
\kappa_{ib} &= (b_0\eta(b_1)b_2)\otimes(b_8^\T\eta(b_7^\T)b_6^\T).
}
We then obtain $\kappa_\pi({\bf a},{\bf b})$ by plugging the previous equations in $\Theta\left(\kappa_{eb} \Theta[\kappa_{re}] \kappa_{ib}^\T\right)$.
\end{example}

\begin{proof}[Proof of Lemma~\ref{Lemma:TypeI}]
We decompose the sum in Def.~\ref{Def:CumulantAP} into three pieces. Fix $p,q,r,s\in[d]$. An appropriate change of variables, and the observation at the end of Section~\ref{Subsection:MCCTDLP}, shows that
\eq{\SUM{\repdc{i'}{t'_1+1}{t'_k}\\\repdc{j'}{t'_1}{t'_k \minus 1}} \SUM{\repdc{i''}{t''_k+1}{t''_1}\\\repdc{j''}{t''_k}{t''_1 \minus 1}} \left[\prod_{w=t'_1}^{t'_k \minus 1} a_w(j'_w,i'_{w+1})\right] \left[\prod_{w=t''_k}^{t''_1 \minus 1} b_{n \minus w}^\T(j''_w,i''_{w+1})\right] \prod_{\underset{\pi_{re}}{u \sim v}} \sigma(i_u,j_u;i_v,j_v)}
equals
\eq{\kappa_{\pi_{re}}({\rm I}_d,\repdc{a}{t'_1}{t'_k \minus 1},{\rm I}_d;{\rm I}_d,\repdc{b}{n \minus t''_1+1}{n \minus t''_k},{\rm I}_d)(i'_{t'_1},j'_{t'_k};j''_{t''_1},i''_{t''_k}),}
which in turn equals $\kappa_{re}(i'_{t'_1},j'_{t'_k};j''_{t''_1},i''_{t''_k})$. As usual, let $j'_0 := p$ and $i'_{n+1} := q$. Observe that
\eqn{eq:Sumeb}{\SUM{\repdc{i'}{1}{t'_1 \minus 1}\\\repdc{j'}{1}{t'_1 \minus 1}} \SUM{\repdc{i'}{t'_k+1}{n}\\\repdc{j'}{t'_k+1}{n}} \left[\prod_{w=0}^{t'_1 \minus 1} a_w(j'_w,i'_{w+1})\right] \left[\prod_{w=t'_k}^{m} a_w(j'_w,i'_{w+1})\right] \prod_{\underset{\pi_{eb}}{u \sim v}} \sigma(i_u,j_u;i_v,j_v)}
can be written as
\eq{\SUM{\repdc{i'}{1}{t'_1 \minus 1}\\\repdc{j'}{1}{t'_1 \minus 1}} \SUM{\repdc{i'}{t'_k+1}{n}\\\repdc{j'}{t'_k+1}{n}} \left[\prod_{w=0}^{t'_1 \minus 1} a_w(j'_w,i'_{w+1})\right] \left[\prod_{w=t'_k}^{m} c_{m \minus w}^\T(j'_w,i'_{w+1})\right] \prod_{\underset{\pi_{eb}}{u \sim v}} \sigma(i_u,j_u;i_v,j_v),}
where $c_{m \minus w}=a_w^\T$ for $t'_k\leq w\leq m$. After a suitable change of variables, Def.~\ref{Def:CumulantDP} shows that \eqref{eq:Sumeb} equals $\kappa_{\pi_{eb}}(\repdc{a}{0}{t'_1 \minus 1};\repdc{c}{0}{m \minus t'_k})(p,i'_{t'_1};q,j'_{t'_k})$, which equals
\eq{\kappa_{\pi_{eb}}(\repdc{a}{0}{t'_1 \minus 1};\repdc{a^\T}{m}{t'_k})(p,i'_{t'_1};q,j'_{t'_k}) = \kappa_{eb}(p,i'_{t'_1};q,j'_{t'_k}).}
Let $j''_0 := s$ and $i''_{n+1} := r$. As with $\pi_{eb}$, we have that
\eq{\SUM{\repdc{i''}{1}{t''_k \minus 1}\\\repdc{j''}{1}{t''_k \minus 1}} \SUM{\repdc{i''}{t''_1+1}{n}\\\repdc{j''}{t''_1+1}{n}} \left[\prod_{w=0}^{t''_k \minus 1} b_{n \minus w}^\T(j''_w,i''_{w+1})\right] \left[\prod_{k=t''_1}^n b_{n \minus w}^\T(j''_w,i''_{w+1})\right] \prod_{\underset{\pi_{ib}}{u \sim v}} \sigma(i_u,j_u;i_v,j_v)}
equals $\kappa_{\pi_{ib}}(\repdc{b^\T}{n}{n \minus t''_k+1};\repdc{b}{0}{n \minus t''_1})(s,i''_{t''_k};r,j''_{t''_1})$. By Lemma~\ref{Lemma:ReversedDPCumulant}, it equals
\eq{\kappa_{\widehat{\pi_{ib}}}(\repdc{b}{0}{n-t''_1};\repdc{b^\T}{n}{n-t''_k+1})(r,j''_{t''_1};s,i''_{t''_k}) = \kappa_{ib}(r,j''_{t''_1};s,i''_{t''_k}).}
Altogether, we obtain that
\al{
\kappa_\pi({\bf a};{\bf b})(p,q;r,s) &= \SUM{i'_{t'_1},j'_{t'_k},i''_{t''_k},j''_{t''_1}} \kappa_{eb}(p,i'_{t'_1};q,j'_{t'_k}) \kappa_{re}(i'_{t'_1},j'_{t'_k};j''_{t''_1},i''_{t''_k}) \kappa_{ib}(r,j''_{t''_1};s,i''_{t''_k})\\
&= \SUM{i'_{t'_1},j'_{t'_k},i''_{t''_k},j''_{t''_1}} \kappa_{eb}(p,i'_{t'_1};q,j'_{t'_k}) \Theta[\kappa_{re}](i'_{t'_1},j''_{t''_1};j'_{t'_k},i''_{t''_k}) \kappa_{ib}^\T(j''_{t''_1},r;i''_{t''_k},s)\\
&= [\kappa_{eb}\Theta[\kappa_{re}]\kappa_{ib}^\T](p,r;q,s)\\
&= \Theta(\kappa_{eb}\Theta[\kappa_{re}]\kappa_{ib}^\T)(p,q;r,s).
}
Since $p,q,r,s\in[d]$ are arbitrary, we conclude that $\kappa_\pi({\bf a};{\bf b}) = \Theta(\kappa_{eb} \Theta[\kappa_{re}] \kappa_{ib}^\T)$.
\end{proof}

In particular, for $a,b\in\MAT_d$,
\eqn{eq:CumulantTI}{\kappa_\pi(a;b) = \Theta\left(\kappa_{\pi_{eb}}(a;a^\T)\Theta[\kappa_{\pi_{re}}({\rm I}_d,\repdc{a}{}{},{\rm I}_d;{\rm I}_d,\repdc{b}{}{},{\rm I}_d)]\kappa_{\widehat{\pi_{ib}}}(b;b^\T)^\T\right).}
As before, we get a rather neat expression when we sum over all $\pi$ in $\calT_\textnormal{I}$.

\begin{proposition}
\label{Prop:SumTI}
Let $X\in\Mat{d}{\calA}$ be a MSO semicircular element with covariance $\sigma$. Then
\eq{\SUM{\pi\in\calT_\textnormal{I}} \kappa_\pi(a^{-1};b^{-1}) = \Theta\left(G_D(a,a^\T)\Theta[H(a,b)]G_D(b,b^\T)^\T\right).}
\end{proposition}

\begin{proof}
By Proposition~\ref{Prop:TIDPSP}, the map $\calT_\textnormal{I}\ni\pi\mapsto(\pi_{eb},\pi_{ib},\pi_{re})\in\DP\times\DP\times\DP^{||}$ is a bijection. In particular,
\eq{\SUM{\pi\in\calT_\textnormal{I}} \kappa_\pi(a^{-1};b^{-1}) = \SUM{\pi_{eb}\in\DP} \SUM{\pi_{ib}\in\DP} \SUM{\pi_{re}\in\DP^{||}} \kappa_{(\pi_{eb},\pi_{ib},\pi_{re})}(a^{-1};b^{-1}).}
By equation \eqref{eq:CumulantTI}, we have that \dsty{\SUM{\pi\in\calT_\textnormal{I}} \kappa_\pi(a^{-1};b^{-1})} equals
\eq{\SUM{\pi_{eb}\in\DP, \pi_{ib}\in\DP \\ \pi_{re}\in\DP^{||}} \Theta\left(\kappa_{\pi_{eb}}(a^{\minus 1};a^{\minus\T})\Theta[\kappa_{\pi_{re}}({\rm I}_d,\repdc{a^{\minus 1}}{}{},{\rm I}_d;{\rm I}_d,\repdc{b^{\minus 1}}{}{},{\rm I}_d)]\kappa_{\widehat{\pi_{ib}}}(b^{\minus 1};b^{\minus\T})^\T\right),}
where $a^{-\T} = (a^\T)^{-1}$. Since $\Theta$ is linear and $\{\hat{\pi} : \pi\in\DP\} = \DP$, we obtain that
\eq{\SUM{\pi\in\calT_\textnormal{I}} \kappa_\pi(a^{-1};b^{-1}) = \Theta\left(G_D(a,a^\T)\Theta[H(a,b)]G_D(b,b^\T)^\T\right),}
as required.
\end{proof}

\subsubsection{Type II annular pairings}

Let $A^\Gamma(p,q;r,s)=A(p,q;s,r)$ for all $A\in\MAT_{d^2}$ and $p,q,r,s\in[d]$. Also, let $\Phi(A)=\Theta(A^\Gamma)$.

\begin{lemma}
\label{Lemma:TypeII}
Let $X\in\Mat{d}{\calA}$ be a MSO semicircular element with covariance $\sigma$. If $\pi\in\AP_{m,n}^{(k)}\cap\calT_\textnormal{II}$ and $\{t'_1,m+t''_1\},\ldots,\{t'_k,m+t''_k\}$ are its through strings with $t'_1<\cdots<t'_k$ and $t''_{s+1}>\cdots>t''_k>t''_1>\cdots>t''_s$ for some $s\in[k]$, then
\eq{\kappa_\pi({\bf a};{\bf b}) = \Theta(\kappa_{eb} \Phi[\kappa_{rr}^\Gamma(\kappa_{oi}\otimes\kappa_{oe})^\Gamma\kappa_{lr}^\Gamma]\kappa_{ib}^\T),}
where
\al{
\kappa_{oi} &:= \kappa_{\pi_{oi}}(\repdc{a}{t'_s}{t'_{s+1} \minus 1}),\\
\kappa_{oe} &:= \kappa_{\widehat{\pi_{oe}}}(\repdc{b}{n \minus t''_k+1}{n \minus t''_1}),\\
\kappa_{rr} &:= \kappa_{\pi_{rr}}({\rm I}_d,\repdc{a}{t'_1}{t'_s \minus 1},{\rm I}_d;{\rm I}_d,\repdc{b}{n \minus t''_1+1}{n \minus t''_s},{\rm I}_d),\\
\kappa_{lr} &:= \kappa_{\pi_{rl}}({\rm I}_d,\repdc{a}{t'_{s+1}}{t'_k \minus 1},{\rm I}_d;{\rm I}_d,\repdc{b}{n \minus t''_{s+1}+1}{n \minus t''_k},{\rm I}_d),\\
\kappa_{eb} &:= \kappa_{\pi_{eb}}(\repdc{a}{0}{t'_1 \minus 1};\repdc{a^\T}{m}{t'_k}),\\
\kappa_{ib} &= \kappa_{\widehat{\pi_{ib}}}(\repdc{b}{0}{n \minus t''_{s+1}};\repdc{b^\T}{n}{n \minus t''_s+1}).
}
\end{lemma}

In this case, $\kappa_\pi({\bf a},{\bf b})$ can be constructed, mutatis mutandis, as in Example~\ref{Example:CumulantTI}. In fact, the next proof is very similar to that of Lemma~\ref{Lemma:TypeI}.

\begin{proof}[Proof of Lemma~\ref{Lemma:TypeII}]
A change of variables shows that
\eq{\SUM{\repdc{i'}{t'_s+1}{t'_{s+1} \minus 1}}\SUM{\repdc{j'}{t'_s+1}{t'_{s+1} \minus 1}} \left[\prod_{w=t'_s}^{t'_{s+1} \minus 1} a_w(j'_w,i'_{w+1})\right] \prod_{\underset{\pi_{oi}}{u \sim v}} \sigma(i_u,j_u;i_v,j_v)}
equals
\eq{\kappa_{\pi_{oi}}(\repdc{a}{t'_s}{t'_{s+1}-1})(j'_{t'_s},i'_{t'_{s+1}}) = \kappa_{oi}(j'_{t'_s},i'_{t'_{s+1}}).}
Recall the property of the single-line cumulants established in Lemma~\ref{Lemma:ReversedSPCumulant}. As with $\pi_{oi}$, we have that
\eq{\SUM{\repdc{i''}{t''_1+1}{t''_k \minus 1}}\SUM{\repdc{j''}{t''_1+1}{t''_k \minus 1}} \left[\prod_{w=t''_1}^{t''_k \minus 1} b_{n \minus w}^\T(j''_w,i''_{w+1})\right] \prod_{\underset{\pi_{oe}}{u \sim v}} \sigma(i_u,j_u;i_v,j_v)}
equals
\al{
\kappa_{\pi_{oe}}(\repdc{b^\T}{n-t''_1}{n-t''_k+1})(j''_{t''_1},i''_{t''_k}) &= \kappa_{\widehat{\pi_{oe}}}(\repdc{b}{n \minus t''_k+1}{n \minus t''_1})(i''_{t''_k},j''_{t''_1}) = \kappa_{oe}^\T(j''_{t''_1},i''_{t''_k}).
}
As in Lemma~\ref{Lemma:TypeI}, we have that
\eq{\SUM{\repdc{i'}{t'_1+1}{t'_s}\\\repdc{j'}{t'_1}{t'_s \minus 1}} \SUM{\repdc{i''}{t''_s+1}{t''_1}\\\repdc{j''}{t''_s}{t''_1 \minus 1}} \left[\prod_{w=t'_1}^{t'_s \minus 1} a_w(j'_w,i'_{w+1})\right] \left[\prod_{w=t''_s}^{t''_1 \minus 1} b_{n \minus w}^\T(j''_w,i''_{w+1})\right] \prod_{\underset{\pi_{rr}}{u \sim v}} \sigma(i_u,j_u;i_v,j_v)}
equals
\eq{\kappa_{\pi_{rr}}({\rm I}_d,\repdc{a}{t'_1}{t'_s \minus 1},{\rm I}_d;{\rm I}_d,\repdc{b}{n \minus t''_1+1}{n \minus t''_s},{\rm I}_d)(i'_{t'_1},j'_{t'_s};j''_{t''_1},i''_{t''_s}) = \kappa_{rr}(i'_{t'_1},j'_{t'_s};j''_{t''_1},i''_{t''_s}),}
and
\eq{\SUM{\repdc{i'}{t'_{s+1}+1}{t'_k}\\\repdc{j'}{t'_{s+1}}{t'_k \minus 1}} \SUM{\repdc{i''}{t''_k+1}{t''_{s+1}}\\\repdc{j''}{t''_k}{t''_{s+1} \minus 1}} \left[\prod_{w=t'_{s+1}}^{t'_k \minus 1} a_w(j'_w,i'_{w+1})\right] \left[\prod_{w=t''_k}^{t''_{s+1} \minus 1} b_{n \minus w}^\T(j''_w,i''_{w+1})\right] \prod_{\underset{\pi_{lr}}{u \sim v}} \sigma(i_u,j_u;i_v,j_v)}
equals
\eq{\kappa_{\pi_{lr}}({\rm I}_d,\repdc{a}{t'_{s+1}}{t'_k \minus 1},{\rm I}_d;{\rm I}_d,\repdc{b}{n \minus t''_{s+1}+1}{n \minus t''_k},{\rm I}_d)(i'_{t'_{s+1}},j'_{t'_k};j''_{t''_{s+1}},i''_{t''_k}).}
Note that, by definition, the last expression equals $\kappa_{lr}(i'_{t'_{s+1}},j'_{t'_k};j''_{t''_{s+1}},i''_{t''_k})$. In an analogous way, we have that
\eq{\SUM{\repdc{i'}{1}{t'_1 \minus 1}\\\repdc{j'}{1}{t'_1 \minus 1}} \SUM{\repdc{i'}{t'_k+1}{n}\\\repdc{j'}{t'_k+1}{n}} \left[\prod_{w=0}^{t'_1 \minus 1} a_w(j'_w,i'_{w+1})\right] \left[\prod_{w=t'_k}^{m} a_w(j'_w,i'_{w+1})\right] \prod_{\underset{\pi_{eb}}{u \sim v}} \sigma(i_u,j_u;i_v,j_v)}
equals $\kappa_{\pi_{eb}}(\repdc{a}{0}{t'_1 \minus 1};\repdc{a^\T}{m}{t'_k})(p,i'_{t'_1};q,j'_{t'_k}) = \kappa_{eb}(p,i'_{t'_1};q,j'_{t'_k})$, and
\eq{\SUM{\repdc{i''}{1}{t''_s \minus 1}\\\repdc{j''}{1}{t''_s \minus 1}} \SUM{\repdc{i''}{t''_{s+1}+1}{n}\\\repdc{j''}{t''_{s+1}+1}{n}} \left[\prod_{w=0}^{t''_s \minus 1} b_{n \minus w}^\T(j''_w,i''_{w+1})\right] \left[\prod_{k=t''_{s+1}}^n b_{n \minus w}^\T(j''_w,i''_{w+1})\right] \prod_{\underset{\pi_{ib}}{u \sim v}} \sigma(i_u,j_u;i_v,j_v)
}
equals
\eq{\kappa_{\widehat{\pi_{ib}}}(\repdc{b}{0}{n \minus t''_{s+1}};\repdc{b^\T}{n}{n \minus t''_s+1})(r,j''_{t''_{s+1}};s,i''_{t''_s}) = \kappa_{ib}(r,j''_{t''_{s+1}};s,i''_{t''_s}).}
Altogether, $\kappa_\pi({\bf a};{\bf b})(p,q;r,s)$ equals
\al{
& \hspace{-20pt} \SUM{i'_{t'_1},j'_{t'_k},i''_{t''_s},j''_{t''_{s+1}}} \hspace{-20pt} \kappa_{eb}(p,i'_{t'_1};q,j'_{t'_k}) \kappa_{ib}(r,j''_{t''_{s+1}};s,i''_{t''_s}) \times \\
& \quad \quad \times \hspace{-20pt} \SUM{j'_{t'_s},i'_{t'_{s+1}},j''_{t''_1},i''_{t''_k}} \hspace{-20pt} \kappa_{rr}(i'_{t'_1},j'_{t'_s};j''_{t''_1},i''_{t''_s}) \kappa_{oi}(j'_{t'_s},i'_{t'_{s+1}}) \kappa_{oe}^\T(j''_{t''_1},i''_{t''_k})  \kappa_{lr}(i'_{t'_{s+1}},j'_{t'_k};j''_{t''_{s+1}},i''_{t''_k}),
}
which in turn equals
\al{
& \SUM{i'_{t'_1},j'_{t'_k},i''_{t''_s},j''_{t''_{s+1}}} \hspace{-20pt} \kappa_{eb}(p,i'_{t'_1};q,j'_{t'_k}) [\kappa_{rr}^\Gamma(\kappa_{oi}\otimes\kappa_{oe}^\T)\kappa_{lr}^\Gamma](i'_{t'_1},j'_{t'_k};i''_{t''_s},j''_{t''_{s+1}}) \kappa_{ib}^\T(j''_{t''_{s+1}},r;i''_{t''_s},s)\\
& \quad \quad = \SUM{i'_{t'_1},j'_{t'_k},i''_{t''_s},j''_{t''_{s+1}}} \hspace{-20pt} \kappa_{eb}(p,i'_{t'_1};q,j'_{t'_k}) \Phi[\kappa_{rr}^\Gamma(\kappa_{oi}\otimes\kappa_{oe}^\T)\kappa_{lr}^\Gamma](i'_{t'_1},j''_{t''_{s+1}};j'_{t'_k},i''_{t''_s}) \kappa_{ib}^\T(j''_{t''_{s+1}},r;i''_{t''_s},s)\\
& \quad \quad = \kappa_{eb} \Phi[\kappa_{rr}^\Gamma(\kappa_{oi}\otimes\kappa_{oe}^\T)\kappa_{lr}^\Gamma]\kappa_{ib}^\T (p,r;q,s)\\
& \quad \quad = \Theta(\kappa_{eb} \Phi[\kappa_{rr}^\Gamma(\kappa_{oi}\otimes\kappa_{oe}^\T)\kappa_{lr}^\Gamma]\kappa_{ib}^\T)(p,q;r,s).
}
Since $p,q,r,s\in[d]$ are arbitrary, we conclude that
\eq{\kappa_\pi({\bf a};{\bf b}) = \Theta(\kappa_{eb} \Phi[\kappa_{rr}^\Gamma(\kappa_{oi}\otimes\kappa_{oe})^\Gamma\kappa_{lr}^\Gamma]\kappa_{ib}^\T),}
as required.
\end{proof}

In particular, $\Theta\left(\kappa_\pi(a;b)\right)$ equals
{\footnotesize\eq{\kappa_{\pi_{eb}}(a;a^\T) \Phi[\kappa_{\pi_{rr}}({\rm I}_d,a,\ldots,{\rm I}_d;{\rm I}_d,b,\ldots,{\rm I}_d)^\Gamma(\kappa_{\pi_{oi}}(a)\otimes\kappa_{\widehat{\pi_{oe}}}(b))^\Gamma \kappa_{\pi_{lr}}({\rm I}_d,a,\ldots,{\rm I}_d;{\rm I}_d,b,\ldots,{\rm I}_d)^\Gamma] \kappa_{\widehat{\pi_{ib}}}(b;b^\T)^\T.}}
As before, we get a rather neat expression when we sum over all $\pi$ in $\calT_\textnormal{II}$.

\begin{proposition}
\label{Prop:SumTII}
Let $X\in\Mat{d}{\calA}$ be a MSO semicircular element with covariance $\sigma$. Then
\eq{\SUM{\pi\in\calT_\textnormal{II}} \kappa_\pi(a^{-1};b^{-1}) = \Theta\left(G_D(a,a^\T)\Phi\left[H(a,b)^\Gamma (G_S(a)\otimes G_S(b))^\Gamma H(a,b)^\Gamma\right]G_D(b,b^\T)^\T\right).}
\end{proposition}

\begin{proof}
Observe that the mappings $\Phi(\cdot)$ and $(\cdot)^\Gamma$ are both linear. Using the bijection established in Proposition~\ref{Prop:TIIDPSP}, the previous lemma implies the result as in the proof of Proposition~\ref{Prop:SumTI}.
\end{proof}

As before, consider the following.

\begin{definition}
Let $X\in\Mat{d}{\calA}$ be a MSO semicircular element with covariance $\sigma$. We define the annular Cauchy transform $G_A:\MAT_d\times\MAT_d\to\MAT_{d^2}$ of $X$ by
\eq{G_A(a,b) = \sum_{\pi\in\AP} \kappa_\pi(a^{-1};b^{-1}).}
\end{definition}

Our main combinatorial theorem is now a simple consequence Propositions~\ref{Prop:SumTI} and \ref{Prop:SumTII}.

\begin{theorem}
\label{Thm:OpVSOCauchyTransform}
Let $X\in\Mat{d}{\calA}$ be a MSO semicircular element with covariance $\sigma$. Then
\eq{G_A(a,b) = \Theta\bigg(G_D(a,a^\T) \bigg\{\Theta[H(a,b)]+\Phi\left[H(a,b)^\Gamma (G_S(a)\otimes G_S(b))^\Gamma H(a,b)^\Gamma\right]\bigg\} G_D(b,b^\T)^\T\bigg).}
\end{theorem}

All the results in this section were derived at the level of formal power series. The purpose of the next section is to extend these results to an analytical level.

\section{Some analytical properties of the MSO Cauchy transform of MSO semicircular elements}
\label{Section:AnalyticProperties}

Assume that $X$ is a MSO semicircular element. By Corollary~\ref{Corollary:MatricialCauchyTransform}, its MSO Cauchy transform $\calG_2$ equals the annular Cauchy transform $G_A$. In particular, Theorem~\ref{Thm:OpVSOCauchyTransform} implies that
\eqn{eq:AnnularCauchyTransform}{\calG_2(a,b) = \Theta\bigg(G_D(a,a^\T) \bigg\{\Theta[H(a,b)]+\Phi\left[H(a,b)^\Gamma (G_S(a)\otimes G_S(b))^\Gamma H(a,b)^\Gamma\right]\bigg\} G_D(b,b^\T)^\T\bigg).}
In this section we explore the analytical properties of the right hand side of this equation. In order to do so, we need to assume an underlying analytic framework.

We assume that our second-order probability space $(\calA,\FI,\FI_2)$ is such that $\calA = \C\langle X(p,q) : p,q\in[d]\rangle$ and that $X=(X(p,q))_{p,q}\in\Mat{d}{\calA}$ is a MSO semicircular element. Also, we assume that there exists a tracial C${}^\ast$-probability space $(\calB,\psi)$ such that $\calA\subset\calB$ and $\FI = \psi|_{\calA}$. Note that even though $\FI$ can be extended to the whole C${}^*$-algebra $\calB$, the bilinear functional $\FI_2$ is only defined on $\calA\times\calA$. In fact, in some classical examples \cite{DiaconisShahshahani1994} the bilinear functional $\FI_2$ cannot be extended continuously to the C${}^*$-algebra generated by $\calA$. We further assume that $X$ and $\MAT_d$ are algebraically free, so that monomials $a_0X \cdots Xa_n$ with $\repdc{a}{0}{n}\in\MAT_d$ are linearly independent from each other whenever they have different lengths $n$. By definition, the domain of $\EOP_2$ contains $\MAT_d \langle X\rangle \times \MAT_d \langle X\rangle$. Observe that the system
\eq{\MAT_d \langle X\rangle:= \textrm{Span}_\mathbb C\{a_0X \cdots Xa_n\colon n\in\mathbb N,a_j\in \MAT_d\}}
is norm-dense in the C${}^*$-algebra generated by $X$ and $\MAT_d$. As with $\FI_2$, the mapping $\EOP_2$ may be unbounded in the norm topology of the C${}^*$-algebra generated by $X$ and $\MAT_d$. Establishing analytical properties of $\calG_2$ amounts to extending $\EOP_2$ beyond $\MAT_d \langle X\rangle\times\MAT_d \langle X\rangle$. The following lemma establishes the bounded behavior of $\EOP_2$ in monomials from $\MAT_d\langle X\rangle$.

\begin{lemma}
If $X\in\Mat{d}{\calA}$ is a MSO semicircular element with covariance $\sigma$, then
\eq{\| \EOP_2(a_0X\cdots Xa_m \otimes b_0X\cdots Xb_n) \| \leq d \left(2d^2\sqrt{\|\Sigma\|}\right)^m \left(2d^2\sqrt{\|\Sigma\|}\right)^n \|a_0\| \cdots \|a_m\| \cdot \|b_0\| \cdots \|b_n\|}
for all $m,n\in\N$ and $\repdc{a}{0}{m},\repdc{b}{0}{n}\in\MAT_d$.
\end{lemma}

\begin{proof}
Recall that $M_{m,n}({\bf a};{\bf b}) = \EOP_2(a_0X\cdots X a_m \otimes b_0X\cdots X b_n)$. Fix $p,q,r,s\in[d]$. From Proposition~\ref{Proposition:MomentCumulantFormula}, we have that $M_{m,n}({\bf a};{\bf b})(p,q;r,s)$ equals
\eq{\sum_{\pi\in\AP_{m,n}} \SUM{\repdc{i'}{1}{m}\\\repdc{j'}{1}{m}} \SUM{\repdc{i''}{1}{n}\\\repdc{j''}{1}{n}} \left[\prod_{w=0}^m a_w(j'_w,i'_{w+1})\right] \left[\prod_{w=0}^n b_{n \minus w}^\T(j''_w,i''_{w+1})\right] \prod_{\underset{\pi}{u \sim v}} \sigma(i_u,j_u;i_v,j_v)}
with $j'_0=p$, $i'_{m+1}=q$, $j''_0=s$, and $i''_{n+1}=r$. Since $|A(i,j)|\leq\|A\|$ for all $A\in\MAT_d$ and all $1\leq i,j\leq d$,
\eqn{eq:EstimateMomentAP}{|M_{m,n}({\bf a};{\bf b})(p,q;r,s)| \leq \left[\prod_{w=0}^m \|a_w\|\right] \left[\prod_{w=0}^n \|b_w\|\right] \|\Sigma\|^{(m+n)/2} d^{2m} d^{2n} |\AP_{m,n}|.}
Recall that \cite[eq. (11)]{MingoSpeicherTan2009}, for every $m,n\geq1$,
\eq{|\AP_{m,n}| = \begin{cases}\frac{mn}{2(m+n)} \binom{m}{m/2}\binom{n}{n/2} & m,n\textnormal{ even},\\
\frac{(m+1)(n+1)}{8(m+n)} \binom{m+1}{(m+1)/2} \binom{n+1}{(n+1)/2} & m,n\textnormal{ odd},\\
0 & \textnormal{otherwise.}\end{cases}}
The standard estimate $\binom{2n}{n} \leq 2^{2n}/\sqrt{2n}$ (e.g., \cite{Robbins1955}) readily implies that \dsty{|\AP_{m,n}|\leq 2^m 2^n}. Plugging this inequality in \eqref{eq:EstimateMomentAP}, we obtain that
\eq{|M_{m,n}({\bf a};{\bf b})(p,q;r,s)| \leq \left(2d^2\sqrt{\|\Sigma\|}\right)^m \left(2d^2\sqrt{\|\Sigma\|}\right)^n \left[\prod_{w=0}^m \|a_w\|\right] \left[\prod_{w=0}^n \|b_w\|\right].}
In particular, $\|M_{m,n}({\bf a};{\bf b})\|_\textnormal{max}$ is bounded by the same quantity and therefore
\eq{\|M_{m,n}({\bf a};{\bf b})\| \leq d \left(2d^2\sqrt{\|\Sigma\|}\right)^m \left(2d^2\sqrt{\|\Sigma\|}\right)^n \|a_0\| \cdots \|a_m\| \cdot \|b_0\| \cdots \|b_n\|,}
as we wanted to prove.
\end{proof}

With the usual abuse of notation, the previous lemma reads as
\eq{\|M_{m,n}(a;b)\| \leq d \left(2d^2\sqrt{\|\Sigma\|}\right)^m \left(2d^2\sqrt{\|\Sigma\|}\right)^n \|a\|^{m+1} \|b\|^{n+1}.}
If $a$ and $b$ are invertible and $\|a^{-1}\|,\|b^{-1}\|<\left(2d^2\sqrt{\|\Sigma\|}\right)^{-1}$, then the series
\eqn{eq:calG2ConvergentPowerSeries}{\calG_2(a,b) = \sum_{m,n\geq1} \EOP_2\left((a^{-1}X)^na^{-1} \otimes (b^{-1}X)^mb^{-1}\right)}
converges in the C${}^*$-norm of $\MAT_d\otimes\MAT_d$. Thus, even though $\calG_2(a,b)$ was defined at the level of formal expressions, it has a well defined meaning as a matricial power series for such $a$ and $b$. By our standing assumptions, the mapping $\EOP_2$ is only defined on $\Mat{d}{\calA}\times\Mat{d}{\calA}$. Nonetheless, \eqref{eq:calG2ConvergentPowerSeries} allows us to define
\eq{\EOP_2\left(\left(a-X\right)^{-1} \otimes \left(b-X\right)^{-1}\right) := \sum_{m,n\geq1} \EOP_2\left((a^{-1}X)^na^{-1},(b^{-1}X)^mb^{-1}\right) = \calG_2(a,b)}
whenever $a,b,a-X,b-X$ are invertible in $\Mat{d}{\calB}$ and $\|a^{-1}\|,\|b^{-1}\|<\left(2d^2\sqrt{\|\Sigma\|}\right)^{-1}$.

In order to lift the equality in \eqref{eq:AnnularCauchyTransform} to an analytic level, we  introduce the following transforms:
\begin{eqnarray*}
\calG(a)&=& \Eop{(a-X)^{-1}} = (\id\otimes\FI)\left((a-X)^{-1}\right),\\
\calG_D(a,b)&=& [\calG(a)\otimes\calG(b)] \left({\rm I}_{d^2}-\Sigma [\calG(a)\otimes\calG(b)]\right)^{-1},\\
\calH(a,b)&=&\left({\rm I}_{d^2}-\Sigma[\calG(a)\otimes\calG(b)]\right)^{-1}\Sigma.
\end{eqnarray*}
Since \ndsty{\calG(a^{-1})=\Eop{(a^{-1}-X)^{-1}}=\Eop{(1-aX)^{-1}}a}, we conclude that $a\mapsto \calG(a^{-1})$ extends analytically to the ball $\{a\in\MAT_d \colon \|a\|<\|X\|^{-1}\}$. We agree to denote this extension by $\calG(a^{-1})$ even when $a$ is not invertible. Since $\calG(a^{-1})\simeq a$ in norm near zero, there exists a neighborhood of zero, say $B_0$, such that the mappings $(a,b)\mapsto\calG_D(a^{-1},b^{-1})$ and $(a,b)\mapsto\calH(a^{-1},b^{-1})$ extend analytically to $B_0\times B_0$. Note that if $a,b\in\MAT_d$ are invertible and $a^{-1},b^{-1}\in B_0$, then $G_S(a)$, $G_D(a,b)$ and $H(a,b)$ are convergent power series with limits $\calG(a)$, $\calG_D(a,b)$ and $\calH(a,b)$, respectively. Therefore, equation \eqref{eq:AnnularCauchyTransform} implies that for such $a$ and $b$ we have that
\eq{\Theta \EOP_2(\left(a-X\right)^{-1} \otimes \left(b-X\right)^{-1})= \calG_D(a,a^\T)\left\{\Theta[\calH(a,b)]+\Phi\left[\calH(a,b)^\Gamma(\calG(a)\otimes\calG(b))^\Gamma\calH(a,b)^\Gamma\right]\right\}\calG_D(b,b^\T)^\T}
in the sense of analytic mappings on $\MAT_d\times\MAT_d$. This equality allows us to extend the domain of definition of $\EOP_2$ to be equal to the domain of the right hand side. In this direction, the main result of this section is the following. Recall that $a^{-\T}=(a^\T)^{-1}$.

\begin{theorem}
\label{Thm:AnalyticalBehavior}
For $a,b\in\MAT_d$ invertible, let
\eq{\calD=\{(z,w)\in\mathbb C^2\colon 1-zaX,1-wbX\text{ invertible}\}.}
The mapping
\al{
(z,w) &\mapsto \calG_D((za)^{-1},(za)^{-\T}) \bigg\{\Theta[\calH((za)^{-1},(wb)^{-1})] + \\
& \quad + \Phi\left[\calH((za)^{-1},(wb)^{-1})^\Gamma(\calG((za)^{-1})\otimes \calG((wb)^{-1}))^\Gamma\calH((za)^{-1},(wb)^{-1})^\Gamma\right]\bigg\} \times\\
& \quad \times \calG_D((wb)^{-1},(wb)^{-\T})^\T
}
is analytic on the connected component of $\calD$ containing $(0,0)$.
\end{theorem}

This theorem immediately leads to the following. By abuse of notation, we let $z$ denote $z{\rm I}_d$.

\begin{corollary}
\label{Corollary:AnalyticitySOCauchyTransform}
The map
\eq{(z,w)\mapsto \calG_D(z,z) \left\{\Theta[\calH(z,w)]+\Phi\left[\calH(z,w)^\Gamma(\calG(z)\otimes\calG(w))^\Gamma\calH(z,w)^\Gamma\right]\right\} \calG_D(w,w)^\T}
is analytic on the set $(\C\bs[-\|X\|,\|X\|])^2$.
\end{corollary}

The rest of this section is devoted to the proof of Theorem~\ref{Thm:AnalyticalBehavior}. We start with two technical lemmas.

\begin{lemma}
Let $E$ be a complex Banach space. Assume that $\Omega\subset\mathbb C$ is an open connected set and $f\colon\Omega\to E$ is a one-to-one analytic function such that $f'(z)\neq0$ for all $z\in\Omega$. If $K\subset\Omega$ is a compact set, then \dsty{\sup_{\begin{smallmatrix}z_1,z_2\in K\\z_1\neq z_2\end{smallmatrix}} \frac{|z_2-z_1|}{\|f(z_2)-f(z_1)\|_E}} is finite.
\end{lemma}

\begin{proof}
In order to reach contradiction, assume that the supremum is unbounded. In that case, there exist sequences $(p_n)_{n\geq1},(q_n)_{n\geq1}\subset K$ such that
\eq{\lim\limits_{n\to\infty} \frac{\|f(p_n)-f(q_n)\|_E}{|p_n-q_n|} = 0.}
Since $K$ is compact, by passing to a subsequence, we may assume that $(p_n,q_n)\to(p,q)\in K^2$. If $p\neq q$, then
\eq{\lim\limits_{n\to\infty} \frac{\|f(p_n)-f(q_n)\|_E}{|p_n-q_n|} = \frac{\|f(p)-f(q)\|_E}{|p-q|}.}
Since $f$ is one-to-one, the right hand side of the previous equation has to be strictly positive. This contradicts the fact that the left hand side of the previous equation is zero. If $p=q$, then the Taylor series expansion of $f$ around $p$ shows that $f'(p)=0$. Since the latter equality contradicts our assumptions, we conclude that the supremum has to be bounded.
\end{proof}

\begin{lemma}
\label{Lemma:ExtensionAnalyticMappings}
Let $D$ be a dense subspace of a Banach space $E$ and let $A:D\to\C^N$ be a linear operator. Assume that $\Omega\subset\mathbb C$ is an open connected set and $f\colon\Omega\to E$ is a one-to-one analytic function such that $f'(z)\neq0$ for all $z\in\Omega$. If $S\subset\Omega$ is discrete in $\Omega$, and $f(\Omega\setminus S)\subset D$, and $A\circ f\colon\Omega\setminus S\to\mathbb C^N$ is analytic, then $A\circ f$ extends analytically to $\Omega$. 
\end{lemma}

\begin{proof}
For notational simplicity, let $g=A \circ f\colon\Omega\setminus S\to\C^N$. Since $S$ is discrete, it is enough to show that $g$ can be analytically extended to a given point $w\in S$.

Since $S$ is discrete, there exists $r\in(0,1)$ such that $\{z\in\C:|z-w| \leq r\} \cap S = \{w\}$. Let $K=\{z\in\mathbb C\colon r/2\leq|z-w|\leq r\}$. Any point $z\in K$ can be expressed uniquely as $z=w+\rho e^{2\pi it}$, $\rho=|z-w|$, $t=\arg(z-w)\in[0,2\pi)$. We start proving that $A$ is Lipschitz on $f(K)$. Take $u_1,u_2\in f(K)$. Since $f$ is one-to-one,  for $k=1,2$ there exists a unique  $z_k\in K$ such that $u_k=f(z_k)$. Since $g$ is analytic on a neighbourhood of $K$, it follows there exists a constant $m=m(K,g)\in[0,+\infty)$ such that  \dsty{m=\sup_{z\in K} \|g'(z)\|_{\C^N}}. 

Connect the two points $z_1,z_2$ via two concatenated paths, $\theta_1,\theta_2$, one being an arc of a circle centered at $w$, the other a segment on a ray starting from $w$. (See Figure \ref{fig:curve_figure}.)
Without loss of generality, 
we write $z_1 = w + \rho_1 e^{2 \pi i t_1}$ and $z_2 = w + \rho_2 e^{2 \pi i t_2}$ with
$\rho_1\leq\rho_2$ and $0\leq t_2-t_1\leq1/2$ (the other cases are treated the same way).  Then $\|g(z_2) - g(z_1)\|_{\C^N}\leq\|g(z_2) - g(\zeta)\|_{\C^N}+\|g(\zeta) - g(z_1)\|_{\C^N}$, where $\zeta=w+\rho_1e^{2\pi it_2}$. To estimate the first term, let $\gamma:\R\to\C$ be given by $\gamma(t) = w+\rho_1e^{2\pi i t}$, so that $\theta_1=\gamma|_{[t_1,t_2]}$. Then

%\begin{figure}[t]
%\includegraphics{curve_figure}
%\caption{\label{fig:curve_figure} The curves $\theta_1$ and $\theta_2$ connecting $z_1$ to $z_2$.}
%\end{figure}

\begin{figure}
  \begin{minipage}[c]{0.25\textwidth}
    \includegraphics[width=\textwidth]{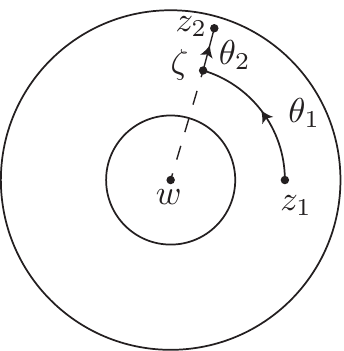}
  \end{minipage}\hfill
  \begin{minipage}[c]{0.7\textwidth}
    \caption{\label{fig:curve_figure}
       The curve $\theta_1$ connecting $z_1$ to $\zeta$ and the curve $\theta_2$  connecting $\zeta$ to $z_2$.}
  \end{minipage}
\end{figure}

%We start proving that $A$ is Lipschitz on $f(\gamma(\R))$. Take $u_1,u_2\in f(\gamma(\R))$. Since $f$ is one-to-one, there exists a unique $z_k\in\gamma(\R)$ such that $u_k=f(z_k)$ for $k=1,2$. It is easy to show that there exists $t_1\in[0,1)$ and $t_2\geq t_1$ with $t_2-t_1\leq 1/2$ such that either $z_1=\gamma(t_1)$ and $z_2=\gamma(t_2)$ or $z_1=\gamma(t_2)$ and $z_2=\gamma(t_1)$. Without loss of generality, assume the former case. In this case,
\al{
\|g(z_1) - g(\zeta)\|_{\C^N} &= \|g(\gamma(t_2)) - g(\gamma(t_1))\|_{\C^N}\\
&\leq \int_{t_1}^{t_2} |\gamma'(s)| \|g'(\gamma(s))\|_{\C^N} \dif s\\
&\leq 2\pi \rho_1 m (t_2-t_1),
}
where $m$ has been defined above. 
%\marginpar{I have changed \\ `$r$' to $\rho_1$ \\ on lines 4 \& 8.}
%\dsty{m=\sup_{z\in\gamma(\R)} \|g'(z)\|_{\C^N}}. By assumption, the function $g$ is analytic on $\gamma(\R)\subset\Omega \setminus S$. Since $\gamma(\R)$ is compact, we have that $m<\infty$. 
Observe that $|z_1-\zeta| =\rho_1|e^{2\pi i (t_2-t_1)} - 1| = 2 \rho_1 \sin(\pi(t_2-t_1)) \geq 4 \rho_1(t_2 - t_1)$. 
%Let $h:[0,1/2]\to\R$ be given by $h(t)=|e^{2\pi i t}-1|$. Observe that $h(0)=0$ and $h(1/2)=2$. A straightforward computation shows that $h''(t) = - \sqrt{2} \pi^2 (1-\cos(2\pi t))^{1/2}$ for all $t\in[0,1/2]$. In particular, $h$ is concave and therefore $h(t)\geq4t$. As a consequence, $4\rho_1(t_2-t_1) \leq |\zeta-z_1|$ and 
Hence
\eq{\|g(\zeta) - g(z_1)\|_{\C^N} \leq \frac{\pi}{2} m |\zeta-z_1|.}
%In particular, we have that
%\al{
%\|A(f(\zeta)) - A(f(z_1))\|_{\C^N} &\leq \frac{\pi}{2} m M \|f(\zeta)-f(z_1)\|_E,
%}
%where \dsty{M = \sup_{\begin{smallmatrix}z,z'\in K\\z\neq z'\end{smallmatrix}} \frac{|z-z'|}{\|f(z)-f(z')\|_E}}. By the previous lemma, $M$ is necessarily finite. 
The second term is majorized as follows:
$$
\|g(z_2) - g(\zeta)\|_{\C^N}\leq\int_0^1|z_2-\zeta|\|g'(sz_2+(1-s)\zeta)\|_{\C^N}\,{\rm d}s\leq{m|z_2-\zeta|}.
$$
Thus, 
$$
\|g(z_2) - g(z_1)\|_{\C^N}\leq\frac{\pi m}{2}|\zeta-z_1|+{m}|z_2-\zeta|.
$$
Note that $|z_1-z_2|\ge\max\{|\zeta-z_1|,|z_2-\zeta|\}$. It follows that
$$
\|g(z_2) - g(z_1)\|_{\C^N}\leq\frac{(2+\pi) m}{2}|z_2-z_1|.
$$
In particular, it follows that 
\eq{\|A(u_2)-A(u_1)\|_{\C^N} \leq \frac{\pi+2}{2} m M \|u_2-u_1\|_E,}
where \dsty{M = \sup_{\begin{smallmatrix}z,z'\in K\\z\neq z'\end{smallmatrix}} \frac{|z-z'|}{\|f(z)-f(z')\|_E}} is finite according to the previous lemma,
and hence $A$ is Lipschitz on $f(K)$.

Consider a sequence $\{w_k\}_k$ converging to $w$. Since $f$ is analytic on $\Omega$, Cauchy's integral formula, see, e.g., Theorem~3.31 in \cite{Rudin1991}, implies that for $\gamma(t) = w + r e^{2 \pi i t}$ we have
\eq{f(w)=\frac{1}{2\pi i} \int_{\gamma} \frac{f(z)}{z-w} \,dz,\quad f(w_k)=\frac{1}{2\pi i} \int_{\gamma} \frac{f(z)}{z-w_k} \,dz,}
for all $k$ sufficiently large. Recall that since $f$ is analytic in the norm topology, the integrals also converge in norm.  More specifically, 
\eq{\frac{1}{2\pi i} \int_{\gamma} \frac{f(z)}{z-w} \,dz = \int_0^1f(w+re^{2\pi i t})\,dt =\lim_{n\to\infty} \frac{1}{n} \sum_{j=1}^nf\left(w+re^{2\pi i\xi_j^{(n)}}\right)}
and \eq{\frac{1}{2\pi i} \int_{\gamma} \frac{f(z)}{z-w_k} \,dz = \int_0^1\frac{f(w+re^{2\pi i t})}{1+\frac{w-w_k}{re^{2\pi i t}}}
\,dt =\lim_{n\to\infty} \frac{1}{n} \sum_{j=1}^n\frac{f\left(w+re^{2\pi i\xi_j^{(n)}}\right)}{1+\frac{w-w_k}{re^{2\pi i \xi_j^{(n)}}}}}
in the norm topology of $E$, for a partition of $[0,1]$ into $n$ equal segments and a choice of $\xi_j^{(n)}$ in each segment. From the Lipschitzianity of $A$ on $f(K)$, we obtain that
\eq{g(w) :=
\lim_{n\to\infty} \frac{1}{n} \sum_{j=1}^n (A\circ f)\left(w+re^{2\pi i\xi_j^{(n)}}\right)}
exists, with a similar statement for $w_k$. 
Note that
\begin{eqnarray*}
\lefteqn{\left\|\frac{1}{n} \sum_{j=1}^nA\left(f\left(w+re^{2\pi i\xi_j^{(n)}}\right)\right)-
\frac{1}{n} \sum_{j=1}^nA\left(\frac{f\left(w+re^{2\pi i\xi_j^{(n)}}\right)}{1+\frac{w-w_k}{re^{2\pi i \xi_j^{(n)}}}}\right)\right\|}\\
& = & \left\|\frac{1}{n} \sum_{j=1}^n\frac{w-w_k}{re^{2\pi i \xi_j^{(n)}}+w-w_k}A\left(f\left(w+re^{2\pi i\xi_j^{(n)}}\right)\right)\right\|\\
& \leq & \frac{1}{n} \sum_{j=1}^n\frac{|w-w_k|}{r-|w-w_k|}\left\|A\left(f\left(w+re^{2\pi i\xi_j^{(n)}}\right)\right)\right\|\\
& < & \frac{2}{r}|w-w_k|\max\{\|A(f(z))\|\colon z\in K\},
\end{eqnarray*}
for all $k\in\mathbb N$ sufficiently large, where we have used the Lipschitzianity of $A$ on $f(K)$ in the last inequality. Thus, $\|g(w_k)-g(w)\|\to0$ as $k\to\infty$.
%By Cauchy's theorem, $g(w)$ does not depend on the simple path taken. Furthermore, a similar lipschitzianity argument shows that $g(w_n) \to g(w)$ for all $w_n\to w$. 
Since, according to our hypothesis, $g(w_k)=A(f(w_k))$, by \cite[10.14]{Rudin1987}, the extension $g:(\Omega \setminus S)\cup\{w\}\to\C^N$ is in fact analytic.
\end{proof}

Now we are in position to prove Theorem~\ref{Thm:AnalyticalBehavior}.

\begin{proof}[Proof of Theorem~\ref{Thm:AnalyticalBehavior}]
Note that
\eq{(a,b) \mapsto
\calG_D(a,a^\T)\left\{\Theta[\calH(a,b)]+\Phi\left[\calH(a,b)^\Gamma(\calG(a)\otimes\calG(b))^\Gamma\calH(a,b)^\Gamma\right]\right\}\calG_D(b,b^\T)^\T}
is well-defined on
\begin{eqnarray*}
\lefteqn{\{(a,b)\in M_d(\mathbb C)\times M_d(\mathbb C)\colon a-X,b-X\text{ invertible, } }\\
& & 1\not\in\sigma(\Sigma \calG(a)\otimes\calG(b)^\T)\cup\sigma(\Sigma \calG(a)\otimes \calG(a)^\T)\cup\sigma(\Sigma \calG(b)\otimes \calG(b)^\T)\},
\end{eqnarray*}
where $\sigma(V)$ denotes the spectrum of the linear operator $V$. As a consequence, this allows us to extend 
$(a, b) \mapsto \Theta( \EOP_2(\left(a-X\right)^{-1}\otimes\left(b-X\right)^{-1}))$ to the same set.

Let $a,b\in\MAT_d$ be invertible. For such $a,b\in\MAT_d$, we define the map
\al{
g \colon (z,w) &\mapsto \calG_D((za)^{-1},(za)^{-\T}) \bigg\{\Theta[\calH((za)^{-1},(wb)^{-1})] + \\
& \quad + \Phi\left[\calH((za)^{-1},(wb)^{-1})^\Gamma(\calG((za)^{-1})\otimes \calG((wb)^{-1}))^\Gamma\calH((za)^{-1},(wb)^{-1})^\Gamma\right]\bigg\} \times\\
& \quad \times \calG_D((wb)^{-1},(wb)^{-\T})^\T.
}
By the discussion before Theorem~\ref{Thm:AnalyticalBehavior}, the mapping $g$ is well-defined on a neighborhood of $(0,0)$. Let $g_w:z\mapsto g(z,w).$ For $w_0$ small enough, the mapping $g_{w_0}$ is well-defined on
\al{\{z\in\C : 1-zaX\textnormal{ invertible},1\not\in\sigma(\Sigma \calG((za)^{-1})\otimes\calG(b^{-1})^\T)\cup\sigma(\Sigma \calG((za)^{-1})\otimes \calG((za)^{-1})^\T)\}.}
We would like to be able to extend $g_{w_0}$ to the possibly larger set $\Omega$ which is the connected component of $\{z\in\mathbb C\colon1-zaX\text{ invertible}\}$ containing zero. Observe that $z\mapsto\det({\rm I}_{d^2}-\Sigma G((za)^{-1})\otimes G(b^{-1})^\T)$ maps $z=0$ to $1$, so that the set of its zeros in $\Omega$ is necessarily discrete. Note that this holds as well for $z\mapsto\det(1_{d^2}-\Sigma G((za)^{-1})\otimes G((za)^{-1})^\T)$. Let us denote by $S$ the union of these two discrete sets. The fact that $g_{w_0}$ extends to $S$, i.e., that $g_{w_0}$ has no singularities in $\Omega$, is an immediate application of the previous lemma with $A(x)=\EOP_2\left(x\otimes\left((w_0b)^{-1}-X\right)^{-1}\right)$, the analytic map $f(z)=\left((za)^{-1}-X\right)^{-1}=z(1-zaX)^{-1}a$, and the discrete set $S$ defined above. Let $z_0$ be in the connected component of $\{z\in\C : 1-zaX\text{ invertible}\}$ containing zero. By the above argument, the mapping $w\mapsto \EOP_2\left(\left((z_0a)^{-1}-X\right)^{-1}\otimes\left((wb)^{-1}-X\right)^{-1}\right)$ is well-defined on a small disc around zero. Mutatis mutandis, we can extend this mapping to the set $\{w\in\C:1-wbX\textnormal{ invertible}\}$. Thus, the map $(z,w)\mapsto \EOP_2\left(\left((za)^{-1}-X\right)^{-1}\otimes\left((wb)^{-1}-X\right)^{-1}\right)$ is analytic on the connected component $\{(z,w)\in{\mathbb C}^2 : 1-zaX,1-wbX\text{ invertible}\}$ containing $(0,0)$.
\end{proof}

\section*{Acknowledgements}

We would like to thank R.~Speicher for pointing out the possibility of using the matricial second-order Cauchy transform of block Gaussian matrices to obtain the second-order Cauchy transform of non-commutative rational functions evaluated on selfadjoint Gaussian matrices.

\printbibliography
\end{document}